\documentclass[journal]{IEEEtran}
\newcommand\figurewidth{0.48}

\usepackage{amsthm}
\makeatletter
\def\ps@IEEEtitlepagestyle{%
  \def\@oddfoot{\mycopyrightnotice}%
  \def\@evenfoot{}%
}
\def\mycopyrightnotice{%
  {\footnotesize This work has been submitted to the IEEE for possible publication. Copyright may be transferred without notice, after which this version may no longer be accessible.\hfill}
  \gdef\mycopyrightnotice{}
}
\usepackage[utf8]{inputenc}
\usepackage{graphicx}
\usepackage[center]{caption}
\graphicspath{{.}}
\usepackage{color}
\usepackage{epsfig}
\usepackage{latexsym}
\usepackage{amssymb}

\usepackage{mathrsfs}
\usepackage{amsfonts, outlines, mathtools, outlines, cancel, hyperref, stackengine, romannum, dsfont, bbm}
\usepackage{amsmath}
\usepackage{makecell} 
\usepackage{multirow}

\usepackage{enumitem}

\setenumerate[1]{label=(\roman*)}

\usepackage{setspace}
\usepackage{subcaption}
\usepackage[ruled,vlined,linesnumbered]{algorithm2e}
\usepackage[english]{babel}
\usepackage[utf8]{inputenc}
\usepackage[T1]{fontenc}
\usepackage{newtxtext,newtxmath}
\usepackage{algorithmic}
\usepackage{scalerel}  

\newtheorem{theorem}{Theorem}
\newtheorem{example}{Example}
\newtheorem{assumption}{Assumption}
\newtheorem{lemma}{Lemma}
\newtheorem{remark}{Remark}

\newtheorem{envdef}{Definition}
\newtheorem{coro}{Corollary}

\DeclareMathSizes{10}{9}{6}{5}
\allowdisplaybreaks

\newcommand{\abs}[1]{\lvert#1\rvert}%
\newcommand{\norm}[1]{\lVert#1\rVert}
\newcommand{\bnorm}[1]{\big\lVert#1\big\rVert}
\newcommand{\Bnorm}[1]{\Big\lVert#1\Big\rVert}

\newcommand{\blkd}[1]{\text{blkd}(#1)}
\newcommand{\diag}[1]{\text{diag}(#1)}
\newcommand{\cl}[1]{\text{cl}(#1)}
\newcommand{\expt}[2]{\mathbb{E}_{#1}[#2]}
\newcommand{\Bexpt}[2]{\mathbb{E}_{#1}\Big[#2\Big]}
\DeclarePairedDelimiter\floor{\lfloor}{\rfloor}
\newcommand\ubar[1]{\stackunder[1.2pt]{$#1$}{\rule{.8ex}{.075ex}}}

\DeclareMathOperator{\dom}{dom}

\DeclareMathOperator{\minimize}{minimize}
\DeclareMathOperator{\maximize}{maximize}
\DeclareMathOperator{\subj}{subject\:to}

\newcommand{\playerN}{\mathcal{N}}

\newcommand{\bone}{\boldsymbol 1}
\newcommand{\bzero}{\boldsymbol 0}

\newcommand{\rset}[2]{\mathbb{R}^{#1}_{#2}}

\newcommand{\nset}[2]{\mathbb{N}^{#1}_{#2}}

\newcommand{\vol}[1]{\text{vol}(#1)}

\DeclareMathOperator{\argmin}{argmin}
\DeclareMathOperator{\argmax}{argmax}

\usepackage{accents}

\DeclareMathOperator{\proj}{Proj}

\usepackage{cite}
\usepackage{textcomp}
\def\BibTeX{{\rm B\kern-.05em{\sc i\kern-.025em b}\kern-.08em
    T\kern-.1667em\lower.7ex\hbox{E}\kern-.125emX}}

\begin{document}

\title{Zeroth-Order Learning in Continuous Games via Residual Pseudogradient Estimates}
\author{Yuanhanqing Huang and Jianghai Hu
\thanks{This work was supported by the National Science Foundation under Grant No. 2014816 and No.2038410. }
\thanks{The authors are with the Elmore Family School of Electrical and Computer Engineering, Purdue University, West Lafayette, IN, 47907, USA (e-mail: huan1282@purdue.edu; jianghai@purdue.edu).}
}

\maketitle

\begin{abstract}
A variety of practical problems can be modeled by the decision-making process in multi-player games where a group of self-interested players aim at optimizing their own local objectives, while the objectives depend on the actions taken by others. 
The local gradient information of each player, essential in implementing algorithms for finding game solutions, is all too often unavailable. 
In this paper, we focus on designing solution algorithms for multi-player games using bandit feedback, i.e., the only available feedback at each player's disposal is the realized objective values. 
To tackle the issue of large variances in the existing bandit learning algorithms with a single oracle call, we propose two algorithms by integrating the residual feedback scheme into single-call extra-gradient methods. 
Subsequently, we show that the actual sequences of play can converge almost surely to a critical point if the game is pseudo-monotone plus and characterize the convergence rate to the critical point when the game is strongly pseudo-monotone. 
The ergodic convergence rates of the generated sequences in monotone games are also investigated as a supplement. 
Finally, the validity of the proposed algorithms is further verified via numerical examples. 
\end{abstract}


\section{Introduction}
Driven by the proliferation of networked engineering systems with competition over common resources, solving decision-making problems in multi-agent systems with competing interests has drawn an exponentially increasing research interest from the systems and control community \cite{li2022confluence}. 
Examples include smart grid management \cite{maharjan2013dependable, zhu2012differential}, wireless and communication networks \cite{han2012game, zhu2012interference}, transportation systems \cite{vu2021fast}, etc. 
Originating from the seminal work \cite{nash1950equilibrium, neumann2007theoryofgame}, game theory provides the theoretical tools and frameworks to model and analyze the interactions and dynamics of self-interested players. 
Specifically, in the Nash equilibrium problem (NEP), each player independently chooses one action from its strategy set, the preferences of which are indicated by a local objective function. 
In addition to depending on the player's own action, the local objective function is also influenced by the actions taken by other players. 
A central problem at the forefront of this field is the design of algorithms or dynamics, through which, the subscribing players can arrive at stationary solutions, such as Nash equilibria (NEs). 

Most of the existing methods such as \cite{mertikopoulos2019learning,yi2019operator,tatarenko2020geometric} leverage the local partial gradient information to do the update, which necessitates the existence of first-order oracles. 
As the partial gradient information depends on the action taken by other players, the computation of the first-order information requires information from all the participants.
The practical limitations on the communication and computation resources prohibit the use of such centralized first-order oracles, especially in a large-scale multi-agent setup. 
In view of this, there arises a stream of effort in designing algorithms for games on networks that distribute the computation of the partial gradient information and only require local communication within the neighborhood \cite{pavel2019distributed,bianchi2022fast,huang2022distributed}. 
Nevertheless, each player needs to maintain local estimates of others' actions, which elicits the scalability issue when confronted with large global strategy spaces. 
Other concerns could be that players are unwilling to disclose their actions, uncertain about their own objective models, or more extremely, completely oblivious to the existence of the game. 
This motivates us to relax the assumption about feedback and consider the bandit setup, where the only feedback information a player can observe is the realized objective function value. 
To be more specific, this group of players will follow a typical online-learning paradigm that unfolds as follows: at each iteration,  every player selects an action, observes the realized objective function value, and updates its action accordingly and the process repeats \cite{cesa2006prediction,hazan2016introduction}. 


\textit{Related Work:} 
The application of the extra-gradient (EG) methods in solving variational inequalities and finding equilibria in games has a long history, and its early derivations can be found in the work \cite{korpelevich1976extragradient,nemirovski2004prox,nesterov2007dual}. 
In the past decades, a considerable number of the stochastic variants of EG have been developed with their convergence properties investigated in detail \cite{juditsky2011solving,kannan2019optimal,iusem2017extragradient}. 
Compared with other approaches such as the forward-backward method \cite{yi2019operator} or the mirror descent method \cite{zhou2017mirror}, EG can guarantee the convergence of the actual sequences, and the associated ergodic average sequences converge at the rate of $O(1/t)$. 
On the other hand, EG doubles the number of queries for first-order information and projection operations per iteration, which could considerably compromise the algorithms' performance, especially in large-scale problems. 
To reduce the query and computation cost induced by the extra step, significant efforts have been devoted to developing single-call variants of EG by substituting one of the queries and the projections with some approximations based on the information available 
\cite{hsieh2019convergence,gidel2018variational, azizian2021last,malitsky2015projected, cui2016analysis}. 

In the field of optimization, the zeroth-order (or derivative-free) methods have been extensively studied, which approximate the absent gradient information via the perturbed function values received from bandit oracles \cite{flaxman2005online}. 
The single-point methods \cite{gasnikov2017stochastic} leverage one oracle query to procure the gradient estimate, which makes them more attractive for implementation though at the cost of large variance. 
On the other hand, the multi-point methods compute the gradient estimate with two or more queries, which keeps the variance under control, yet complicates the implementation, especially in a time-varying environment. 
To reap the benefits from both, Zhang et al. \cite{zhang2022new} considered a residual-feedback scheme to control the estimation variance. 
The proposed scheme only uses a single query per iteration and matches the performance of two-point zeroth-order methods. 
Inspired by the idea of extremum seeking control, Chen et al. \cite{chen2022improve} developed a novel high/low-pass filter single-point method that further improves the dependency of convergence rates on the problem dimensions. 

As for the literature about learning in games with bandit feedback, Bravo et al. \cite{bravo2018bandit} designed a single-point bandit learning process via the simultaneous perturbation stochastic approximation approach. 
The proposed algorithm has been proved to converge a.s. in games that satisfy diagonal strict concavity and possess an $O(t^{-1/3})$ asymptotic convergence rate when the game is strongly monotone. 
Tatarenko et al. \cite{tatarenko2020bandit} extended the scope of games to merely monotone cases via Tikhonov regularization and developed a no-regret single-point learning algorithm that works in the single timescale, where four decaying sequences should be tuned properly to ensure the convergence. 
In a more recent work \cite{tatarenko2022rate}, Tatarenko et al. introduced a single-point and a two-point bandit learning approaches and proved that if the game is strongly monotone, the convergence rates for them are $O(t^{-1/2})$ and $O(t^{-1})$, respectively. 
Besides, Lin et al. \cite{lin2021optimal} focused on the improvement of solutions for multi-player games to achieve optimal regret and a faster convergence rate. 
They developed a mirror descent variant of the barrier-based family of bandit learning algorithms and proved that this variant is no-regret and converges at a rate of $O(t^{-1/2})$.
However, the scope of this work is limited to strongly monotone games.

\textit{Our Contributions:} 
In view of all the above, our paper tries to address the question of whether or not we can combine the merits of single-call EG and residual feedback from zeroth-order optimization and improve the performance of single-point bandit learning in games. 
Motivated by this, two algorithms are proposed by integrating the idea of residual feedback into optimistic mirror descent and reflected gradient descent, respectively. 
First, we complement the existing results in \cite{zhang2022new} by showing that, in the field of multi-player games, the residual estimation is an unbiased estimate of a smoothed version of the pseudogradient. 
Subsequently, we establish a uniform constant upper bound for the variance of the residual estimate when it is applied in the proposed algorithms. 
Secondly, we prove that if the game satisfies the pseudo-monotone plus assumption, the actual sequences of play generated by the proposed algorithms converge a.s. to a critical point of it. 
Compared with \cite{bravo2018bandit}, our convergence results are obtained under more relaxed regularity assumptions. 
In addition, we analyze the ergodic sequences generated by the proposed algorithms in monotone games and show that they converge at a rate of $O(t^{-(1/2 - \epsilon)})$ for some positive constant $\epsilon$ that can be made arbitrarily small by properly tuning the step size and the exploration radius. 
Lastly, we focus on the convergence speeds of the proposed algorithms in strongly pseudo-monotone games and show that the actual sequences of play can converge to the critical points at a rate of $O(t^{-(1-\epsilon)})$, which considerably accelerate the learning process compared with the existing methods. 
Moreover, the performances of the solution algorithms are empirically compared via the multi-building thermal control problem in Section~\ref{sect:thrm-ctrl}, which illustrates that the proposed algorithms enjoy faster convergence and demonstrate less estimation variance. 

\begin{table}
\begin{tabular}{||c|c|c|c||}
\hline
  \multirow{2}{*}{Ref.}  & Regularity for   & \multicolumn{2}{c|}{Convergence Rate Under} \\ 
  & Convergence & \multicolumn{2}{c|}{Strong Monotonicity} \\
\hline\hline
    \cite{bravo2018bandit}  & Strictly monotone & \multicolumn{2}{c|}{$O(t^{-1/3})$}\\
\hline
    \cite{tatarenko2020bandit} & Merely monotone  & \multicolumn{2}{c|}{N/A} \\
\hline
    \multirow{2}{*}{\cite{tatarenko2022rate}}  & \multirow{2}{*}{Strongly monotone} & Single-query & Two-query \\
    \cline{3-4} & & $O(t^{-1/2})$ & $O(t^{-1})$ \\
\hline
    \cite{lin2021optimal} & Strongly monotone & \multicolumn{2}{c|}{$O(t^{-1/2})$} \\
\hline
    \makecell{Alg.~\ref{alg:optm-mrdesc} \\ and \ref{alg:refl-desc}} & \makecell{Pseudo-monotone \\ plus} & \multicolumn{2}{c|}{$O(t^{-(1-\epsilon)})$} \\
\hline
\end{tabular}
\caption{Bandit learning algorithms for multi-player games: monotonicity assumption and last-iterate convergence rate}
\end{table}

\textit{Organization:} 
In Section~\ref{sect:prob-form}, we formally formulate the multi-player games under study, with the solution concepts, some basic definitions, and assumptions included. 
Moreover, we briefly introduce mirror maps, the associated concepts, and the motivation to use them in this work. 
In Section~\ref{sect:algms}, we introduce the residual feedback scheme to leverage and propose two single-point bandit algorithms. 
The bias and variance of the pseudogradient estimation error are analyzed, which serve as important lifting tools for the later proof. 
Subsequently, in Section~\ref{sect:convg-res}, we state and prove three main convergence results for the proposed algorithms under different regularity assumptions and metrics. 
In Section~\ref{sect:simu}, to demonstrate the theoretical findings and the effectiveness of the proposed algorithms in practical applications, we include three numerical examples: a portfolio optimization problem, a parameter learning problem in linear models, and an optimal thermal management problem in buildings.  
Section~\ref{sect:conclu} concludes the paper and highlights potential extensions and applications. 

\textit{Basic Notations:} 
For a set of matrices $\{V_i\}_{i \in S}$, we let $\blkd{V_1, \ldots, V_{|S|}}$ or $\blkd{V_i}_{i \in S}$ denote the diagonal concatenation of these matrices, $[V_1, \ldots, V_{|S|}]$ their horizontal stack, and $[V_1; \cdots; V_{|S|}]$ their vertical stack. 
For a set of vectors $\{v_i\}_{i \in S}$, $[v_i]_{i \in S}$ or $[v_1; \cdots; v_{|S|}]$ denotes their vertical stack. 
For a vector $v$ and a positive integer $i$, $[v]_i$ denotes the $i$-th entry of $v$. 
Denote $\rset{}{+} \coloneqq [0, +\infty)$, $\rset{}{++} \coloneqq (0, +\infty)$, and $\nset{}{+} \coloneqq \nset{}{} \backslash \{0\}$. 
We let $\norm{\cdot}_2$ represent the Euclidean norm, $\norm{\cdot}$ a general norm, and $\norm{\cdot}_*$ its dual.
For a set $\mathcal{S}$, let $\mathbbm{1}_{\mathcal{S}}$ denote the indicator function for this set, i.e., $\mathbbm{1}_{\mathcal{S}}(x) = 1$ if $x \in \mathcal{S}$ and $0$ otherwise. 
The notation $N_{\mathcal{S}}(x)$ denotes the normal cone to the set $\mathcal{S} \subseteq \rset{n}{}$ at the point $x$: if $x \in \mathcal{S}$, then $N_\mathcal{S}(x) \coloneqq \{u \in \rset{n}{} \mid \sup_{z \in \mathcal{S}} \langle u, z-x \rangle \leq 0 \}$; otherwise, $N_\mathcal{S}(x) \coloneqq \varnothing$. 
Let $\cl{\mathcal{S}}$ denote the closure of set $\mathcal{S}$, $\text{int}(\mathcal{S})$ the interior, and $\partial \mathcal{S}$ the boundary. 
If $\mathcal{S} \in \rset{n}{}$ is a closed and convex set, the map $\proj_\mathcal{S}:\rset{n}{} \to \mathcal{S}$ denotes the projection onto $\mathcal{S}$, i.e., $\proj_\mathcal{S}(x) \coloneqq \argmin_{v \in \mathcal{S}} \norm{v - x}_2$. 

\section{Preliminaries}\label{sect:prob-form}

\subsection{Problem Formulation}

Throughout this paper, we consider an $N$-player game $\mathcal{G}$ consisting of a finite set of participants $\mathcal{N} \coloneqq \{1, \ldots, N\}$. 
Each player $i$ selects its action $x^i$ from the individual strategy space $\mathcal{X}^i \subseteq \rset{n^i}{}$. 
Aggregating over all players, we use $\mathcal{X} \coloneqq \prod_{i \in \playerN} \mathcal{X}^i \subseteq \rset{n}{}$ and $x = [x^i]_{i \in \playerN} \in \mathcal{X}^i$ to represent the strategy space and action profile of the whole game, where $n \coloneqq \sum_{i \in \playerN} n^i$. 
For later notational simplicity, we write the stack of the actions of other players as $x^{-i} \coloneqq [x^j]_{j \in \mathcal{N}^{-i}} \in \mathcal{X}^{-i} \subseteq \rset{n^{-i}}{}$, 
where 
$\mathcal{N}^{-i} \coloneqq \playerN \backslash \{i\}$, 
$\mathcal{X}^{-i} \coloneqq \prod_{j \in \mathcal{N}^{-i}} \mathcal{X}^j$, 
and $n^{-i} \coloneqq \sum_{j \in \mathcal{N}^{-i}}n^j$. 
Also, we write $x \coloneqq [x^i; x^{-i}]$, regardless of the indices among players. 
Under a specific action profile $x \in \mathcal{X}$, an objective value $J^i(x^i; x^{-i})$ will be induced to each player $i$, which determines player $i$'s preference for different actions. 
From the perspective of each self-interested player $i$, it aims to solve the following local optimization problem: 
\begin{align}
\minimize_{x^i \in \mathcal{X}^i} J^i(x^i; x^{-i}). 
\end{align}
Although the local optimization problem is restricted to the feasible set $\mathcal{X}^i$, each player $i$ is assumed to be allowed to take its actions from the action space $\mathcal{X}^i_a$ and receives the corresponding objective values, where the action space $\mathcal{X}^i_a$ can be slightly larger than the strategy space $\mathcal{X}^i$, i.e. $\mathcal{X}^i \subseteq \mathcal{X}^i_a$. 
Stacking the individual action spaces yields the group action space $\mathcal{X}_a \coloneqq \prod_{i \in \mathcal{N}} \mathcal{X}^i_a$.
We make the following blanket assumptions regarding the regularity of the objective function $J^i$, the strategy space $\mathcal{X}^i$, and the action space $\mathcal{X}^i_a$, which are typical in the literature of zeroth-order learning or optimization. 
\begin{assumption}\label{asp:objt-set}
For each player $i$, the local objective function $J^i$ is continuously differentiable in $x$ over the action space $\mathcal{X}_a$.
Moreover, its individual strategy space $\mathcal{X}^i$ is compact and convex and the action space $\mathcal{X}^i_a$ is compact and has a non-empty interior. 
\end{assumption}

\subsection{Game Regularization and Solution Concepts}

To facilitate the later discussion of the regularity assumptions, we introduce the so-called pseudogradient operator $F: \mathcal{X}_a \to \rset{n}{}$ which characterizes the first-order information of $\mathcal{G}$ and is defined as the direct product of partial gradients:
\begin{align}
F(x) \coloneqq \prod_{i \in \mathcal{N}} [\nabla_{x^i} J^i(x^i; x^{-i})]. 
\end{align}
We first make the following assumption concerning the Lipschitz property of $F$ to contend with the absence of first-order information in the zeroth-order setup. 
\begin{assumption}\label{asp:lipschitz}
The pseudogradient $F$ is Lipschitz continuous on $\mathcal{X}_a$ with the constant $L$, i.e., for any $x, x^\prime \in \mathcal{X}_a$, we have
\begin{align}
    \norm{F(x) - F(x^\prime)}_* \leq L\norm{x - x^\prime}. 
\end{align}
Moreover, for each $i$, the operator $\nabla_{x^i}J^i: \mathcal{X}_a \to \rset{n_i}{}$ enjoys a smaller Lipschitz constant $L^i$.
\end{assumption}

Nash equilibrium (NE) is a standard solution concept for non-cooperative games, which is defined as a decision profile resilient to arbitrary unilateral deviations. 
Formally, a decision profile $x_* = [x^{i}_*; x^{-i}_*] \in \mathcal{X}$ is an NE of $\mathcal{G}$ if 
\begin{align}\label{eq:ne}
J^i(x^i_*; x^{-i}_*) \leq J^i(x^i; x^{-i}_*), \; \forall x^i \in \mathcal{X}^i, \; \text{for all } i \in \mathcal{N}. 
\end{align} 
In the sequel, we focus on a more relaxed solution concept than NEs, called critical points (CPs) \cite[Sec.~2.2]{mertikopoulos2022learning}, whose definition is formulated using (Stampacchia) variational inequalities (VIs) and the pseudogradient $F$ defined above. 
\begin{envdef}\label{def:cps} (Critical Points)
A decision profile $x_* \in \mathcal{X}$ is a critical point of the non-cooperative game $\mathcal{G}$ if it is a solution to the associated VI, i.e., 
\begin{align}\label{eq:cps}
\langle F(x_*), x - x_*\rangle \geq 0, \; \forall x \in \mathcal{X}. 
\end{align}
\end{envdef}
We make the blanket assumption that the games discussed in this work admit at least one critical point inside $\mathcal{X}$. 
A well-known result is that CPs coincide with NEs when $J^i$ is convex and continuously differentiable in $x^i$ for all $i$ \cite[Sec.~1.4.2]{facchinei2003finite}. 
Let $\mathcal{X}_*$ denote the set of critical points for $\mathcal{G}$.
Another commonly-used concept in the literature to measure the inaccuracy of a candidate solution $x_\star = [x^{i}_\star; x^{-i}_\star] \in \mathcal{X}$ in a variational form is the following merit function: 
\begin{align}\label{eq:merit-func}
    \text{Err}_{\mathcal{X}}(x_\star) \coloneqq \max_{x \in \mathcal{X}} \langle F(x), x_\star - x\rangle. 
\end{align}
Much of the literature on continuous games examines games that possess (strongly) monotone pseudogradient, which are referred to as (strongly) monotone games \cite{scutari2010convex}. 
Note that for monotone games, $\text{Err}_{\mathcal{X}}(x_*) \geq 0$ for all $x_* \in \mathcal{X}$ and $\text{Err}_{\mathcal{X}}(x_*) = 0$ if and only if $x_*$ is a solution of the VI under study. 
In this work, we shift the scope and proceed with several different classes of games \cite[Def.~2.3.9]{facchinei2003finite}, the regularities of which are described below. 
\begin{envdef}\label{def:regu}
The game $\mathcal{G}$ is 
\begin{outline}[enumerate]
\1 pseudo-monotone if for any action profiles $x, y \in \mathcal{X}$, 
$\langle F(y), x - y\rangle \geq 0 \implies \langle F(x), x - y\rangle \geq 0$;
\1 pseudo-monotone plus if it is pseudo-monotone and for any action profiles $x, y \in \mathcal{X}$, $\langle F(y), x - y\rangle \geq 0 \;\text{and}\; \langle F(x), x - y\rangle = 0 \implies F(x) = F(y)$; 
\1 strictly pseudo-monotone if for any action profiles $x, y \in \mathcal{X}$, 
$\langle F(y), x - y\rangle \geq 0 \implies \langle F(x), x - y\rangle \geq 0$ with equality if and only if $x = y$;
\1 $\mu$-strongly pseudo-monotone if for any action profiles $x, y \in \mathcal{X}$, $\langle F(y), x - y\rangle \geq 0 \implies \langle F(x), x - y\rangle \geq \mu\norm{x - y}^2$;
\1 strictly coherent if for any critical point $x_* \in \mathcal{X}_*$, $\langle F(x), x - x_*\rangle > 0$, $\forall x \in \mathcal{X} \backslash \mathcal{X}_*$, where $\mathcal{X}_*$ is the set of critical points; 
\1 with pseudoconvex potential $\Phi$ if $F$ is the gradient of a pseudoconvex function $\Phi$, i.e., $F=\nabla\Phi$ and for any $x, y \in \mathcal{X}$, $\langle \nabla \Phi(x), y-x\rangle \geq 0 \implies \Phi(y) \geq \Phi(x)$.  
\end{outline}
\end{envdef}
A brief remark about $(\romannum{6})$ is that the differentiable potential $\Phi$ is pseudoconvex on $\mathcal{X}$ if and only if $\nabla \Phi$ is pseudomonotone on $\mathcal{X}$ \cite[Thm.~3.1]{karamardian1976complementarity}. 

\subsection{Mirror Maps between Banach Spaces}
To streamline the flow of our paper, in this subsection, we discuss our motivation to leverage mirror maps and introduce the basic concepts related to them. 
When confronted with an optimization problem, if its objective function and the associated constraint set are well-behaved in a Euclidean space, we can leverage the projected gradient descent method and $\ell_2$ norm usually serves as an efficient distance metric for measurement. 
Nevertheless, for more general situations, more general Banach spaces $\mathcal{B}$'s may turn out to be the desirable ambient spaces, the associated norm metrics of which are not derived from inner products. 
For example, for mixed-strategy games, the local feasible set for each player is described by a probability simplex, and $\ell_1$ norm is conventionally utilized in this case. 
Unlike Hilbert spaces, whose dual spaces are isometric to themselves, directly implementing the gradient descent in a (primal) Banach space $\mathcal{B}$ no longer makes sense, considering that the gradient of the objective function sits inside its dual $\mathcal{B}^*$. 
Fortunately, the idea of mirror descents and mirror maps, first introduced in \cite{nemirovskij1983problem}, can handle this inconsistency by mapping the point in $\mathcal{B}$ to $\mathcal{B}^*$, performing the gradient update in $\mathcal{B}^*$, and finally mapping the updated point back to $\mathcal{B}$. 
For more detailed discussions and examples, the interested reader is referred to \cite[Ch.~4]{bubeck2014theory}\cite{juditsky2022unifying}. 

We start by introducing the definition of distance-generating functions (DGFs) in this work. 
With a slight abuse of notation, a function $\psi: \dom{\psi} \to \rset{}{}$ with $\dom{\psi} \subseteq \rset{n}{}$ is a DGF in the ambient space $\mathcal{B} = (\rset{n}{}, \norm{\cdot})$ if it satisfies the following three properties: 
(\romannum{1}) $\psi$ is differentiable and $\tilde{\mu}$-strongly convex for some constant $\tilde{\mu} > 0$; 
(\romannum{2}) $\nabla \psi(\dom{\psi}) = \rset{n}{}$; 
(\romannum{3}) $\cl{\dom{\psi}} \supseteq \mathcal{X}$ and $\lim_{x \to \partial(\dom{\psi})}\norm{\nabla\psi(x)}_* = +\infty$, i.e., its gradient diverges on the boundary of $\dom{\psi}$. 
With DGFs in hand, the mirror map $\nabla \psi^*$ from $\mathcal{B}^*$ to $\mathcal{B}$ can be defined as: 
\begin{align}\label{eq:mirror-map}
    \nabla \psi^*(z) = \argmax_{x \in \mathcal{X}} \{\langle z, x\rangle - \psi(x)\}, 
\end{align}
where $\mathcal{X}$ denotes the primal feasible set;  
$\psi^*$ from $\mathcal{B}^*$ to $\rset{}{}$ is the convex conjugate of $\psi$, i.e., $\psi^*(z) = \max_{x \in \mathcal{X}} \{\langle z,x \rangle - \psi(x)\}$; 
the expression $\nabla \psi^*$ denotes the subgradient of $\psi^*$; 
and \eqref{eq:mirror-map} follows from Danskin's Theorem. 
The pseudo-distance induced by $\psi$ is the so-called Bregman divergence: 
\begin{align}
D(p, x) = \psi(p) - \psi(x) - \langle \nabla \psi(x), p - x\rangle, \forall p, x \in \dom{\psi}. 
\end{align}
Due to the $\tilde{\mu}$-strong convexity of $\psi$, Bregman divergence can be bounded below by $D(p, x) \geq \tilde{\mu}/2\norm{p - x}^2$. 
Analogous to the projected gradient descent with the $\ell_2$ norm, the prox-mapping $P_{x, \mathcal{X}}: \mathcal{B}^* \to \dom{\psi} \cap \mathcal{X}$ for some fixed $x \in \dom{\psi} \cap \mathcal{X}$ is induced through the Bregman divergence as: 
\begin{align}\label{eq:prox-map}
P_{x, \mathcal{X}}(y) = \argmin_{x^\prime \in \mathcal{X}}\{ \langle y, x - x^\prime\rangle + D(x^\prime, x)\}. 
\end{align}
We refer the readers to Lemma~\ref{le:distgen-func} for the properties of mirror maps and prox-mappings. 
Note that, for a Lipschitz-continuous and convex function $f$, it can be minimized via the mirror descent iteration given by $x_{k+1} \in \nabla \psi^*(\nabla\psi(x_k) - \gamma_k\nabla f(x_k)) = P_{x_k, \mathcal{X}}(-\gamma_k\nabla f(x_k))$, where $(\gamma_k)_{k \in \nset{}{}}$ is a proper sequence of step sizes \cite[Sec.~4.2]{bubeck2014theory}\cite{juditsky2022unifying}. 

Although $D(p, x_k) \to 0$ implies $x_k \to p$ by the fact that $D(p, x) \geq \tilde{\mu}/2\norm{p - x}^2$, it does not come naturally that the converse holds by the construction above, leaving the level sets of $D(p, \cdot)$ short of indicating neighborhoods of $p$ \cite{mertikopoulos2018optimistic}. 
For posterity, the following mild assumption is made regarding the DGF chosen and the corresponding Bregman divergence. 
\begin{assumption}\label{asp:breg-recip}(Bregman Reciprocity)
The DGF $\psi$ chosen satisfies that when $x_k \to p$, we have $D(p, x_k) \to 0$.
\end{assumption}

\section{Two Single-Call Extra-Gradient Algorithms and the Associated Zeroth Order Variants}\label{sect:algms}

\subsection{Extra-gradient Family}

To procure critical points in non-cooperative games as given in \eqref{eq:cps}, we consider the generalized extra-gradient family of algorithms. 
In this family, at each iteration $k$, the algorithms keep updating two states, i.e., the base state $X_k$ and the leading state $X_{k+1/2}$. 
Given a step-size sequence $(\gamma_k)_{k \in \nset{}{}}$ and two sequences $(\tilde{F}_{k})_{k \in \nset{}{}}$ and $(\tilde{F}_{k+1/2})_{k \in \nset{}{}}$ which are related to the first-order/pseudogradient information regarding the base and leading states respectively, the extra-gradient schemes comprise of the following two steps: 
\begin{align}\label{eq:egrad}
\begin{split}
X_{k+1/2} = P_{X_{k}, \mathcal{X}_1}(-\gamma_k\tilde{F}_{k}), \;
X_{k+1} = P_{X_{k}, \mathcal{X}_2}(-\gamma_k\tilde{F}_{k+1/2}). 
\end{split}
\end{align}
To illustrate some feasible choices of the first-order sequences and the sets $\mathcal{X}_1$ and $\mathcal{X}_2$, we provide the following few concrete examples in literature. 
Let $\xi_k$ and $\xi_{k+1/2}$ denote some unbiased random noise with their second-order moments bounded. 

\begin{example}
(Stochastic Mirror Prox (SMP) \cite{juditsky2011solving}) Consider a general ambient space $(\rset{n}{}, \norm{\cdot})$. 
Setting $\tilde{F}_{k} = F(X_k) + \xi_k$, $\tilde{F}_{k+1/2} = F(X_{k+1/2}) + \xi_{k+1/2}$, and $\mathcal{X}_1 = \mathcal{X}_2 = \mathcal{X}$ yields the stochastic mirror prox algorithm: 
$X_{k+1/2} = P_{X_k, \mathcal{X}}(- \gamma_k(F(X_k) + \xi_k))$, $X_{k+1} = P_{X_k, \mathcal{X}}(- \gamma_k(F(X_{k+1/2}) + \xi_{k+1/2}))$.
\end{example}

\begin{example}\label{exp:omd}
(Optimistic mirror descent (OMD) \cite{gidel2018variational, azizian2021last}) Given a general ambient space $(\rset{n}{}, \norm{\cdot})$, let $\tilde{F}_{k} = F(X_{k-1/2}) + \xi_{k-1/2}$, $\tilde{F}_{k+1/2} = F(X_{k+1/2}) + \xi_{k+1/2}$, $\mathcal{X}_1 = \rset{n}{}$, and $\mathcal{X}_1 = \mathcal{X}_2 = \mathcal{X}$, respectively. 
Then we obtain the following updates:
$X_{k+1/2} = P_{X_k, \mathcal{X}}(- \gamma_k(F(X_{k-1/2}) + \xi_{k-1/2}))$, $X_{k+1} = P_{X_k, \mathcal{X}}(- \gamma_k(F(X_{k+1/2}) + \xi_{k+1/2}))$.
\end{example}

\begin{example}\label{exp:rgd}
(Reflected gradient descent (RGD) \cite{malitsky2015projected, cui2016analysis}) Consider the Euclidean ambient space $(\rset{n}{}, \norm{\cdot}_2)$. 
By letting $\tilde{F}_{k} = 1/\gamma_k(X_{k-1} - X_{k})$, $\tilde{F}_{k+1/2} = F(X_{k+1/2}) + \xi_{k+1/2}$, $\mathcal{X}_1 = \rset{n}{}$, and $\mathcal{X}_2 = \mathcal{X}$, we can ground \eqref{eq:egrad} to the explicit iterations as follows: 
$X_{k+1/2} = 2X_k - X_{k-1}$, $X_{k+1} = \proj_{\mathcal{X}}[X_k - \gamma_k(F(X_{k+1/2}) + \xi_{k+1/2})]$. 
\end{example}

With its deterministic counterpart first introduced in \cite{nemirovski2004prox}, the SMP is one of the most widely studied extra-gradient methods. 
Nevertheless, it requires two queries to the first-order oracle and two prox-mappings per iteration, which makes its implementation costly. 
The RGD leverages the reflection of $X_{k-1}$ in $X_k$ rather than evaluating the gradient at $X_k$, which only requires a single Euclidean projection and single oracle call per iteration. 
Yet extending this strategy beyond the Euclidean scope would require extra regularity assumptions on the DGF to use. 
The implementation of OMD involves two prox-mappings but a single oracle call by updating the leading state using the first-order information from the last iteration. 
As a prelude to the next two subsections, we note here that our zeroth-order learning algorithms leverage RGD and OMD as the backbone and the first-order information to estimate at each iteration $k$ is the pseudogradient queried at the leading state $X_{k+1/2}$.

\subsection{Residual Pseudogradient Estimate}

The first-order information at the query point $X_{k+1/2}$, i.e., $F(X_{k+1/2})$, plays an essential role in the OMD and RGD to procure critical points. 
Yet, all too often, it is unavailable and each player $i$ needs to estimate its own partial gradient information $\nabla_{x^i} J^i(X_{k+1/2})$ based on the observed objective values. 
To enable bandit learning, we let each player randomly sample a query direction $u_{k}^{i}$ from the unit sphere in the $n^i$-dimensional space, with the query radius $\delta_k$ at the $k$-th iteration. 
Let $u_k \coloneqq [u_k^i]_{i \in \playerN}$.
For player $i$, the partial gradient estimate candidate at the $k$-th iteration is given by 
\begin{align}\label{eq:grad-hp-est}
\frac{n^i}{\delta_k}\Big(J^i(X_{k+1/2} + \delta_ku_{k}) - J^i(X_{k-1/2} + \delta_{k-1}u_{k-1})\Big)u^i_{k}. 
\end{align}
Stacking \eqref{eq:grad-hp-est} across this group of players gives us the so-called residual pseudogradient estimate. 
Furthermore, to ensure the local feasibility after perturbing the query point, we assume the existence of a closed ball with center $p_i$ and radius $r_i$ within each $\mathcal{X}^i_a$ and apply the following feasibility adjustment steps:
\begin{align*}
\hat{X}^i_{k+1/2} &= X^i_{k+1/2} + \delta_k(u^i_{k} - \frac{1}{r^i}(X^i_{k+1/2} - p^i)) \\
& = (1 - \frac{\delta_k}{r^i})X^i_{k+1/2} + \frac{\delta_k}{r^i}(p^i + r^iu^i_{k}) = \bar{X}^i_{k+1/2} + \delta_ku^i_{k},
\end{align*}
where $\bar{X}^i_{k+1/2} = (1 - \frac{\delta_k}{r^i})X^i_{k+1/2} + \frac{\delta_k}{r^i}p^i$ and we let $\delta_k < r^i$ for all $k$ to make sure that the above is a feasible convex combination of two points inside $\mathcal{X}^i_a$. 
Substituting the perturbed action in \eqref{eq:grad-hp-est} with the adjusted version, we obtain the following estimate
\begin{align}\label{eq:grad-hp-est-adj}
\tag{RPG}
G^i_{k} \coloneqq \frac{n^i}{\delta_k}\Big(J^i(\hat{X}_{k+1/2}) - J^i(\hat{X}_{k-1/2})\Big)u^i_k, 
\end{align}
where $\hat{X}_{k+1/2} \coloneqq [\hat{X}^i_{k+1/2}]_{i \in \mathcal{N}}$ for each $k \in \nset{}{+}$. 
To later characterize the residual pseudogradient estimate \eqref{eq:grad-hp-est-adj}, we construct a filtration as follows: for each $k \in \nset{}{+}$, let $\mathcal{F}_k \coloneqq \sigma\{X_0, u_1, \ldots, u_{k-1}\}$. 
Note that while $X_{k+1/2}, \bar{X}_{k+1/2} \in \mathcal{F}_k$, $\hat{X}_{k+1/2} \notin \mathcal{F}_k$. 
A widely used function in the literature of zeroth-order learning is the $\delta$-smoothed objective function, defined as:
\begin{align}
\tilde{J}^i_{\delta}(x^i; x^{-i}) \coloneqq \frac{1}{\mathbb{V}^i} \int_{\delta \mathbb{S}_{-i}}\int_{\delta \mathbb{B}_i} J^i(x^i + \Tilde{\tau}^i; x^{-i} + \tau^{-i})d\Tilde{\tau}^i d\tau^{-i},
\end{align}
where $\mathbb{S}_{-i} \coloneqq \prod_{j \in \mathcal{N}^{-i}} \mathbb{S}_j \subseteq \rset{n^{-i}}{}$ with each $\mathbb{S}_j$ representing a unit sphere centered at the origin within $\rset{n_j}{}$; $\mathbb{B}_i$ denotes the unit ball centered at the origin inside $\rset{n_i}{}$; $\mathbb{V}^i \coloneqq \vol{\delta \mathbb{B}_i} \cdot \vol{\delta \mathbb{S}_{-i}}$ is the volume constant. 
Note that $\tilde{J}^i_{\delta}(x^i; x^{-i})$ can be interpreted as the mean value of the local objective function $J^i$ within the region $\delta\mathbb{B}_i \times \delta\mathbb{S}_{-i}$. 
Formally, we have the following lemma to relate \eqref{eq:grad-hp-est-adj} to the function $\tilde{J}^i_{\delta}$. 
\begin{lemma}\label{le:unbiased}
Suppose that Assumption~\ref{asp:objt-set} holds. 
Then at each iteration $k \in \nset{}{+}$, $\nabla_{x^i} \tilde{J}^i_{\delta_k}(\bar{X}_{k+1/2})$ is a version of the conditional expectation $\expt{}{G^i_k \mid \mathcal{F}_k}$, i.e., 
\begin{align}
\nabla_{x^i} \tilde{J}^i_{\delta_k}(\bar{X}_{k+1/2}) = \expt{}{G^i_k \mid \mathcal{F}_k} \;\text{a.s.},\; \forall i \in \mathcal{N}. 
\end{align}
\begin{proof}
See Appendix~\ref{appd:rpg}.
\end{proof}
\end{lemma}
In other words, even though it is straightforward to observe that $G^i_k$ is a biased estimate of the true pseudogradient at the point $\bar{X}_{k+1/2}$, conditioning on $\mathcal{F}_k$, $G^i_k$ is an unbiased estimate of the $\delta$-smoothed function $\tilde{J}^i_{\delta_k}$ at $\bar{X}_{k+1/2}$. 
With the above observation in hand, the error induced by the estimate \eqref{eq:grad-hp-est-adj} can be decomposed into systematic error $B^i_k$ and stochastic error $V^i_k$, i.e., 
\begin{align}\label{eq:est-decomp}
\begin{split}
& G^i_k = \nabla_{x^i}J^i(X_{k+1/2}) + \\
& \underbrace{\big(G^i_k - \nabla_{x^i} \tilde{J}^i_{\delta_k}(\bar{X}_{k+1/2})\big)}_{V^i_k} 
+ \underbrace{\big(\nabla_{x^i} \tilde{J}^i_{\delta_k}(\bar{X}_{k+1/2}) - \nabla_{x^i}J^i(X_{k+1/2})\big)}_{B^i_k}. 
\end{split}
\end{align}
Recalling the construction of the filtration $(\mathcal{F}_k)_{k \in \nset{}{+}}$, the systematic error $B^i_k$ is $\mathcal{F}_k$-measurable while the stochastic error $V^i_{k} \in \mathcal{F}_{k+1}$ is not $\mathcal{F}_k$-measurable. 
Let $G_k \coloneqq [G^i_k]_{i \in \playerN}$, and similarly for $B_k$, $V_k$, and other variables. 
The codomains of the random variables $B_k$ and $V_{k}$ are characterized by the lemma below. 

\begin{lemma}\label{le:codom-err}
Suppose that Assumption~\ref{asp:objt-set} holds.
Then for each iteration $k$, $\norm{B_k}_* \leq \alpha_B\delta_k$ for some constant $\alpha_B$ depending on the Lipschitz constant $\tilde{L}$ and the geometry of $\mathcal{X}_a$. 
In addition, $\norm{V_k}_* \leq \alpha_V/(\delta_k)^2$ for some constant $\alpha_V$ depending on the objectives $J^i$'s and the geometry of $\mathcal{X}_a$. 
\end{lemma}
\begin{proof}
See Appendix~\ref{appd:rpg}.
\end{proof}

Such conclusions for the codomains of the errors $B^i_k$ and $V^i_k$ are typical, especially for the single-point zeroth-order methods \cite{bravo2018bandit, tatarenko2022rate, tatarenko2020bandit}. 
As a prelude, in the next subsection, we will present that the codomain of $\norm{V_k}^2_*$ can be bounded by employing proper sequences of step size and query radius, an outcome of the appealing feature of the RPG. 

\subsection{Zeroth-Order Learning with Single-Call Extra-Gradient methods}

We now present two zeroth-order learning algorithms by incorporating the RPG into the extra-gradient schemes OMD and RGD. 
Before proceeding to the algorithm details, we introduce some notations: 
let each player $i$ be endowed with a local DGF $\psi^i$; 
for this group of players, we in addition denote $\psi(x) \coloneqq \sum_{i \in \mathcal{N}} \psi^i(x^i)$; 
$\hat{J}^i_k$ denotes the realized objective function value of player $i$ at the $k$-th iteration. 

\subsubsection{Optimistic Mirror Descent Method}
Combining the learning scheme of OMD in Example~\ref{exp:omd} with the RPG machinery, we readily obtain the dual vectors leveraged at each iteration $k$, i.e., $\tilde{F}_{k+1/2} = [G^i_k]_{i \in \playerN}$ and $\tilde{F}_{k} = [G^i_{k-1}]_{i \in \playerN}$, respectively. 
Separating the OMD updates player-wisely gives the associated payoff-based learning algorithms in Algorithm~\ref{alg:optm-mrdesc}. 
Note that in Algorithm~\ref{alg:optm-mrdesc}, the action space $\mathcal{X}_a$ coincides with the strategy space $\mathcal{X}$. 

\begin{algorithm}
\SetAlgoLined
\caption{Zeroth-Order Learning of CPs Based on Optimistic Mirror Descent (Player $i$)}
\label{alg:optm-mrdesc}
\textbf{Initialize:} $X^i_0 = X^i_{1/2} = X^i_1 \in \mathcal{X}^i \cap \dom{\psi^i}$ arbitrarily; 
$\hat{J}^i_0 = J^i(X^i_{1/2}; X^{-i}_{1/2})$;
$G^i_{0} = \bzero_{n^i}$; 
$p^i, r^i$ to be the center and radius of an arbitrary ball within the set $\mathcal{X}^i$\;
\textbf{At the $k$-th iteration ($k \in \nset{}{+}$)}: \\
\qquad $X^i_{k+1/2} \leftarrow P_{X^i_k, \mathcal{X}^i}(-\gamma_kG^i_{k-1})$\;
\qquad Randomly sample the direction $u^i_k$ from $\mathbb{S}^{n^i}$\; 
\qquad $\hat{X}^i_{k+1/2} \leftarrow (1 - \frac{\delta_k}{r^i})X^i_{k+1/2} + \frac{\delta_k}{r^i}(p^i + r^i u^i_k)$ \;
\qquad Take action $\hat{X}^i_{k+1/2}$ and observe the realized objective function value $\hat{J}^i_{k} \coloneqq J^i(\hat{X}^i_{k+1/2}; \hat{X}^{-i}_{k+1/2})$\;
\qquad $G^i_k \leftarrow \frac{n^i}{\delta_k}(\hat{J}^i_k - \hat{J}^i_{k-1})u^i_k$\; 
\qquad $X^i_{k+1} \leftarrow P_{X^i_k, \mathcal{X}^i}(-\gamma_kG^i_{k})$\;
\textbf{Return:} $\{\hat{X}^i_{k+1/2}\}_{i \in \playerN}$.
\end{algorithm}

\subsubsection{Reflected Mirror Descent}

Extending the Euclidean setup for RGD in Example~\ref{exp:rgd}, we consider applying the reflected strategy in a space with general norm $(\rset{n}{}, \norm{\cdot})$. 
As a result, the vector difference $1/\gamma_k(X_{k-1} - X_k)$ in the primal space no longer works, and it should be replaced by some reflected proxy in the dual space. 
A potential candidate is $\tilde{F}_k = 1/\gamma_k(\nabla \psi(X_{k-1}) - \nabla \psi(X_{k}))$, and the leading state is updated by $X_{k+1/2} = P_{X_{k}, \rset{n}{}}(-\gamma_k\tilde{F}_k)$. 
The prox-mapping $P_{X_{k}, \rset{n}{}}$ can be reduced to unconstrained problems with strongly convex objectives, which usually enjoy closed-form solutions or at least projection-free solutions with exponential convergence rates. 
For later convergence analysis, we impose the additional regularity that the group DGF $\psi$ is norm-like. 
\begin{assumption}\label{asp:distgenfunc-lipsc}
The group DGF $\psi$ is $\tilde{L}$-smooth on $\mathcal{X}_a$, i.e., for arbitrary $x_a$ and $x_b$ in $\mathcal{X}_a$, 
\begin{align*}
    \psi(x_a) \leq \psi(x_b) + \langle \nabla \psi(x_b), x_a - x_b\rangle + \frac{\tilde{L}}{2} \norm{x_a - x_b}^2. 
\end{align*}
An equivalent condition is that $\nabla \psi: \mathcal{X}_a \to \rset{n}{}$ is $\tilde{L}$-Lipschitz: 
\begin{align*}
    \langle \nabla \psi(x_a) - \nabla \psi(x_b), x_a - x_b\rangle \leq \tilde{L} \norm{x_a - x_b}^2. 
\end{align*}
As a result, for each player $i \in \playerN$, its DGF $\psi^i$ is $\tilde{L}^i$-smooth with the constant $\tilde{L}^i \leq \tilde{L}$. 
\end{assumption}

The zeroth-order learning algorithm with the reflected mirror descent (RMD) is summarized in Algorithm~\ref{alg:refl-desc}. 
We let $\mathcal{X}_R$ denote the reflected space of $\mathcal{G}$, i.e., the range of $X_{k+1/2}$ in Algorithm~\ref{alg:refl-desc}.  
A few remarks are in order concerning the relationship among $\mathcal{X}$, $\mathcal{X}_a$, and $\mathcal{X}_R$. 
In the RMD framework, the leading action $X_{k+1/2} \in \mathcal{X}_R$ can fall outside the strategy space $\mathcal{X}$ but should always sit inside $\mathcal{X}_a$, where the necessary regularity assumptions hold. 
As such, compared with OMD, the approximation strategy in RMD possesses better computational efficiency yet imposes more regularity. 
The optimality condition for the update of leading states suggests that $\nabla \psi(X_{k+1/2}) - \nabla \psi(X_k) = \nabla \psi(X_k) - \nabla \psi(X_{k-1})$. 
By the $\tilde{\mu}$-strong convexity and the $\tilde{L}$-smoothness assumed in Assumption~\ref{asp:distgenfunc-lipsc}, $\tilde{\mu}\norm{X_{k+1/2} - X_k} \leq \tilde{L}\norm{X_{k} - X_{k-1}}$ with $X_{k}$ and $X_{k-1} \in \mathcal{X}$, which implies that the reflected set $\mathcal{X}_R$ is bounded.
Furthermore, for any iteration $k$, $\norm{X_{k+1/2} - X_{k}} \leq L \gamma_{k-1}/\Tilde{\mu} \cdot \norm{G_{k-1}}$, where the boundedness of the random variable $\norm{G_k}_*$ will be established in Lemma~\ref{le:bdd-stoch-err} and $\mathcal{X}_R$ can be made arbitrarily close to $\mathcal{X}$ by choosing $\gamma_k$ sufficiently small. 
Altogether, we should have $\mathcal{X} \subseteq \mathcal{X}_R \subseteq \mathcal{X}_a$ and $\mathcal{X}_R \subseteq \cl{\dom{\psi}}$. 

\begin{algorithm}
\SetAlgoLined
\caption{Zeroth-Order Learning of CPs Based on Reflected Mirror Descent (Player $i$)}
\label{alg:refl-desc}
\textbf{Initialize:} 
$X^i_0 = X^i_{1/2} = X^i_1 \in \mathcal{X}^i \cap \dom{\psi^i}$ arbitrarily; 
$\hat{J}^i_0 = J^i(X^i_{1/2}; X^{-i}_{1/2})$; 
$p^i, r^i$ to be the center and radius of an arbitrary ball within the reflected set $\mathcal{X}^i_a$\;
\textbf{At the $k$-th iteration ($k \in \nset{}{+}$)}: \\
\qquad $X^i_{k+1/2} \leftarrow P_{X^i_k, \rset{n^i}{}} ( - (\nabla \psi^i(X^i_{k-1}) - \nabla \psi^i(X^i_{k})))$ \;
\qquad Randomly sample the direction $u^i_k$ from $\mathbb{S}^{n^i}$\; 
\qquad $\hat{X}^i_{k+1/2} \leftarrow (1 - \frac{\delta_k}{r^i})X^i_{k+1/2} + \frac{\delta_k}{r^i}(p^i + r^i u^i_k)$ \;
\qquad Take action $\hat{X}^i_{k+1/2}$ and observe the realized objective function value $\hat{J}^i_{k} \coloneqq J^i(\hat{X}^i_{k+1/2}; \hat{X}^{-i}_{k+1/2})$\;
\qquad $G^i_k \leftarrow \frac{n^i}{\delta_k}(\hat{J}^i_k - \hat{J}^i_{k-1})u^i_k$\; 
\qquad $X^i_{k+1} \leftarrow P_{X^i_k, \mathcal{X}^i}(- \gamma_kG^i_k)$\;
\textbf{Return:} $\{\hat{X}^i_{k+1/2}\}_{i \in \playerN}$.
\end{algorithm}

The lemma below shows that the proposed algorithms can maintain the bounded dual norm of the stochastic error $V_k$ by properly tuning the decaying rate of step size  slightly faster than that of query radius. 

\begin{lemma}\label{le:bdd-stoch-err}
Suppose that Assumption~1 holds. 
In addition, the monotonically decreasing sequences of step size $(\gamma_k)_{k \in \nset{}{+}}$ and the query radius $(\delta_k)_{k \in \nset{}{+}}$ satisfy: $\lim_{k\to \infty} \gamma_k = 0$, $\lim_{k \to\infty} \delta_k = 0$, $\sum_{k \in \nset{}{}} \gamma_k = \infty$, and $\lim_{k \to \infty} \gamma_k / \delta_k = 0$. 
Consider the RPG $(G_k)_{k \in \nset{}{+}}$ generated by Algorithms~\ref{alg:optm-mrdesc} and \ref{alg:refl-desc} and Algorithm~\ref{alg:refl-desc} is executed with $\psi$ satisfying Assumption~\ref{asp:distgenfunc-lipsc}.
Then for each iteration $k$, $\norm{V_k}^2_* \leq C_V$ for some constant $C_V$. 
\end{lemma}
\begin{proof}
See Appendix~\ref{appd:rpg}.
\end{proof}


\section{Convergence Properties of the Proposed Algorithms}\label{sect:convg-res}

In this section, we will determine the critical points convergence properties of Algorithms~\ref{alg:optm-mrdesc} and \ref{alg:refl-desc} in non-cooperative games under several different regularity conditions. 
The analysis will be on the asymptotic convergence properties, the ergodic convergence rate, and the convergence rate of the sequence of realized actions. 

\subsection{Almost-Sure Convergence of the Proposed Algorithms in Pseudo-monotone Plus Cases}\label{sect:convg}

We begin with the asymptotic convergence analysis for the games that possess pseudo-monotone plus pseudogradient. 

\begin{theorem}\label{thm:asconvg}
Consider a pseudo-monotone plus game $\mathcal{G}$. 
Suppose that the players of $\mathcal{G}$ follow Algorithms~\ref{alg:optm-mrdesc} and Assumptions~\ref{asp:objt-set} to \ref{asp:breg-recip} hold, or players all perform Algorithm~\ref{alg:refl-desc} and Assumptions~\ref{asp:objt-set} to \ref{asp:distgenfunc-lipsc} hold. 
Moreover, the step size $(\gamma_k)_{k \in \nset{}{+}}$ and query radius $(\delta_k)_{k \in \nset{}{+}}$ are monotonically decreasing and satisfy 
\begin{align}\label{eq:merelm-params}
\begin{split}
& \sum_{k \in \nset{}{+}} \gamma_k = \infty, \; 
\sum_{k \in \nset{}{+}} \gamma_k^2 < \infty, \\
& \sum_{k \in \nset{}{+}}\gamma_k\delta_k < \infty, \; \lim_{k \to \infty} \gamma_k/\delta_k = 0. 
\end{split}
\end{align}
Then the sequence of play $(\hat{X}_{k+1/2})_{k \in \nset{}{+}}$ converges to one of the CP $x_*$ almost surely.
\end{theorem}
\begin{proof}
See Appendix~\ref{appd:as-convg}.
\end{proof}

Inspired by the discussion about SMP in \cite{kannan2019optimal}, we extend the convergence results in Theorem~\ref{thm:asconvg} to accommodate several variants of pseudo-monotone plus games given in Def.~\ref{def:regu}, with the results formally stated in the corollary below. 
Before proceeding, we briefly clarify the interrelations between the various classes in Def.~\ref{def:regu}: $(\romannum{5})$ and $(\romannum{6})$ are neither sub-classes nor super-classes of $(\romannum{2})$, i.e., the pseudo-monotone plus class, while $(\romannum{3})$ and $(\romannum{4})$ impose stronger regularity and are sub-classes of $(\romannum{2})$ and $(\romannum{5})$. 

\begin{coro}\label{coro:variant-asconvg}
Consider a game $\mathcal{G}$ with its pseudogradient $F$ satisfying any of Def.~\ref{def:regu} $(\romannum{3})$-$(\romannum{6})$, while the other settings are the same as those considered in Theorem~\ref{thm:asconvg}. 
Then the sequence of play $(\hat{X}_{k+1/2})_{k \in \nset{}{+}}$ converges to one of the CP $x_*$ a.s.
\end{coro}
\begin{proof}
See Appendix~\ref{appd:as-convg}.
\end{proof}


\subsection{Convergence Rates of Ergodic Average in Merely Monotone Cases}

The convergence properties in the last subsection are asymptotic in nature. 
To characterize the speeds of convergence of the proposed iterations, we will focus on the generated ergodic sequences and quantify their inaccuracy through the metric function \eqref{eq:merit-func}, in the spirit of ergodic convergence analysis of monotone VIs. 
The ergodic average of the iteration's actual sequence of play is defined as
\begin{align}\label{eq:erg-avg}
    \check{X}_k \coloneqq \sum^{k}_{t=1} \gamma_t \hat{X}_{t+1/2} / \sum^{k}_{t=1} \gamma_t.
\end{align}

\begin{theorem}\label{thm:erg-rate}
Consider a merely monotone game $\mathcal{G}$. 
The other settings are the same as the ones given in Theorem~\ref{thm:asconvg}. 
Then the ergodic average $\check{X}_k$ satisfies, 
\begin{align}
\expt{}{\text{Err}_{\mathcal{X}}(\check{X}_k)} \leq \frac{\expt{}{\max_{p \in \mathcal{X}} D(p, X_1)} + M}{\sum_{t=1}^k \gamma_t}, \; \forall k \in \nset{}{+}
\end{align}
where $M$ is a constant that depends on the properties of $\mathcal{G}$ and the specific choices of $\gamma_k$ and $\delta_k$.
\end{theorem}
\begin{proof}
See Appendix~\ref{appd:ergd-convg}. 
\end{proof}

\subsection{$O(1/k^{1-\epsilon})$ Convergence Rate of the Proposed Algorithms in Strongly Pseudo-Monotone Cases}

To study the convergence rate of the realized sequences of play, aside from the previous assumptions, we impose on $F$ the strong pseudo-monotonicity requirement. 
It has been proved in \cite[Thm.~2.1]{kim2016qualitative} that a strongly pseudo-monotone VI admits a unique solution $x_*$, which allows us to leverage the distance between the sequence and the unique solution to measure the convergence rate. 
In addition, we restrict to the case where the Bregman divergence is norm-like, i.e., Assumption~\ref{asp:distgenfunc-lipsc} holds. 
Under these conditions, we obtain the following global convergence result for Algorithms~\ref{alg:optm-mrdesc} and \ref{alg:refl-desc}. 

\begin{theorem}\label{thm:strmon-eucl-convgrate}
Consider a strongly pseudo-monotone game $\mathcal{G}$. 
Suppose that Assumptions \ref{asp:objt-set} to \ref{asp:distgenfunc-lipsc} hold and the players of $\mathcal{G}$ all follow either Algorithm~\ref{alg:optm-mrdesc} or \ref{alg:refl-desc} with their step sizes and query radius chosen as $\gamma_k = c_\gamma / (k + b_\gamma)^{a_\gamma}$ and $\delta_k = c_\delta / (k + b_\delta)^{a_\delta}$, respectively. 
Moreover, if $a_\gamma$ and $a_\delta$ satisfy $0 < a_\delta < a_\gamma < 1$ and $a_\gamma + a_\delta > 1$, then the sequence of actions of play enjoys the convergence rate below:
\begin{align}
\expt{}{\norm{\hat{X}_{k+1/2} - x_*}^2} \leq \frac{M_1}{k^{a_\gamma + a_\delta - 1}} + \frac{M_2}{k}, \forall k > K
\end{align}
where $x_*$ is the unique CP of $\mathcal{G}$; 
$M_1$, $M_2$, and $K$ denote some constants that depend on the properties of $\mathcal{G}$ and the specific choices of $\gamma_k$ and $\delta_k$. 
\end{theorem}
\begin{proof}
See Appendix~\ref{appd:strg-mono-rate}. 
\end{proof}

Theorem~\ref{thm:strmon-eucl-convgrate} serves as the main explicit convergence rate analysis result for the realized actions. 
Theoretically, by letting $a_\gamma$ approach $1$ and $a_\delta$ approach $a_\gamma$ both from the left, the proposed algorithms can then achieve the convergence rate $\expt{}{\norm{\hat{X}_{k+1/2} - x_*}^2} = O(1/k^{1 - \epsilon})$ for some $\epsilon$ arbitrarily close to $0$. 
In the practical implementation, it is advisable to adopt a more conservative choice of $a_\gamma$ and $a_\delta$ rather than the one discussed above to circumvent bad transient behavior caused by undesirable $M_1$, $M_2$, and $K$. 
This statement is empirically supported through numerical experiments in Section~\ref{sect:thrm-ctrl}.

\section{Case Study and Numerical Simulations}\label{sect:simu}


\subsection{Portfolio Optimization}
In the single-agent portfolio optimization problem, an agent selects the best portfolio or asset distribution and aims to maximize its expected return\cite{liu2012one}. 
With a slight abuse of notation, assume there exist $N$ assets, whose rates of return are denoted by $\xi \coloneqq [\xi_1; \cdots; \xi_N]$ and are normally distributed, i.e., $\xi \sim \mathcal{N}(\mu, \Sigma)$ with the mean vector $\mu \coloneqq [\mu_1; \cdots; \mu_N] \in \rset{N}{+}$ and the covariance matrix $\Sigma \in \mathcal{S}^N_{++}$. 
We let $z \coloneqq [z_1; \cdots; z_N]$ represent the strategy or asset distribution of the agent, and accordingly, the total rate of return is $\eta = \xi^Tz \sim \mathcal{N}(\mu^Tz, z^T\Sigma z)$. 
Given the expected rate of return $r > 0$, the agent seeks to solve: $\maximize_{z} P(\eta \geq r) = \Phi(\frac{\mu^Tz - r}{\sqrt{z^T\Sigma z}})$, 
subject to $\bone_{N}^T z = 1$, $\mu^Tz - r \geq 0$, and $0 \leq z \leq 1$, where $\Phi(a) \coloneqq \int_{-\infty}^a \frac{1}{\sqrt{2\pi}}\exp(-\frac{\tau^2}{2})d\tau$ denotes the distribution function of standard normal distribution. 
This problem can be further reformulated by noting that $\Phi$ is monotonically increasing and conducting a simple coordinate transform $\varphi: \rset{N-1}{} \to \rset{N}{}$ with $\varphi: x \mapsto [x_1, \cdots, x_{N-1}, 1 - \bone_{N-1}^T x ]$ to procure a feasible set with a non-empty interior. 
Formally, the reformulated problem can be described as follows:
\begin{align}
\begin{split}
& \minimize_{x} J(x) = \frac{r - \mu^T \varphi(x)}{\sqrt{\varphi(x)^T\Sigma\varphi(x)}} \\
& \subj \; x \in \mathcal{X} \coloneqq \{x \in \rset{N-1}{} \mid  0 \leq x \leq 1, \bone_{N-1}^T x \leq 1,  \\
& \qquad\qquad\qquad  [\mu_N - \mu_1; \cdots; \mu_N - \mu_{N-1}]^Tx \leq \mu_N - r\}. 
\end{split}
\end{align}
We claim that the objective function $J(x) = J_n(x)/J_d(x)$ is pseudoconvex on $\mathcal{X}$ where the numerator and denominator $J_n(x) = r - \mu^T \varphi(x)$ and $J_d(x) = \sqrt{\varphi(x)^T\Sigma\varphi(x)}$ are both convex. 
For any $x \in \mathcal{X}$ and $x^\prime \in \mathcal{X}$ that satisfy $\langle \nabla J(x^\prime), x - x^\prime \rangle \geq 0$, we have $\langle \frac{1}{J_d(x^\prime)}\big(\nabla J_n(x^\prime) - \frac{J_n(x^\prime)}{J_d(x^\prime)}\nabla J_d(x^\prime)\big), x - x^\prime \rangle \geq 0$. 
Since $J_n(x) - J_n(x^\prime) \geq \langle \nabla J_n(x^\prime), x - x^\prime\rangle$, $J_d(x) - J_d(x^\prime) \geq \langle \nabla J_d(x^\prime), x - x^\prime\rangle$, $J_n$ is always non-positive on $\mathcal{X}$, and $J_d$ always positive, we finally deduce that $J(x) \geq J(x^\prime)$ and $J$ is pseudo-convex on $\mathcal{X}$. 

In the numerical simulation, the agent has $N=6$ assets to invest, the mean $\mu$ and covariance $\Sigma$ of which are randomly sampled and visualized in Fig.~\ref{fig:portfolio} (a) and (b). 
The expected rate of return $r$ is set as the average of $\mu_i$'s. 
We let the step size and query radius be of the form $\gamma_k = 1/(k+2\times10^3)^{a_{\gamma}}$ and $\delta_k = 1/(k+2\times10^3)^{a_{\delta}}$. 
The performance metrics include the relative updating distance $\norm{X_{k+1} - X_k}_2$ and the difference between the current value $J(X_k)$ and the optimal value $J^*$ obtained via \cite{kannan2019optimal}.
We illustrate the rolling averages of these metrics using solid lines with a window size of $100$ and the original fluctuations with semi-transparent curves in Fig.~\ref{fig:portfolio}. 
Note that in the experiments, RPGs are computed with different numbers of queries, i.e., $q=1, 5, 10$, while in the previous sections, we focus on analyzing the single-query ($q=1$) case.
\begin{figure}
    \centering
    \includegraphics[width=\figurewidth\textwidth]{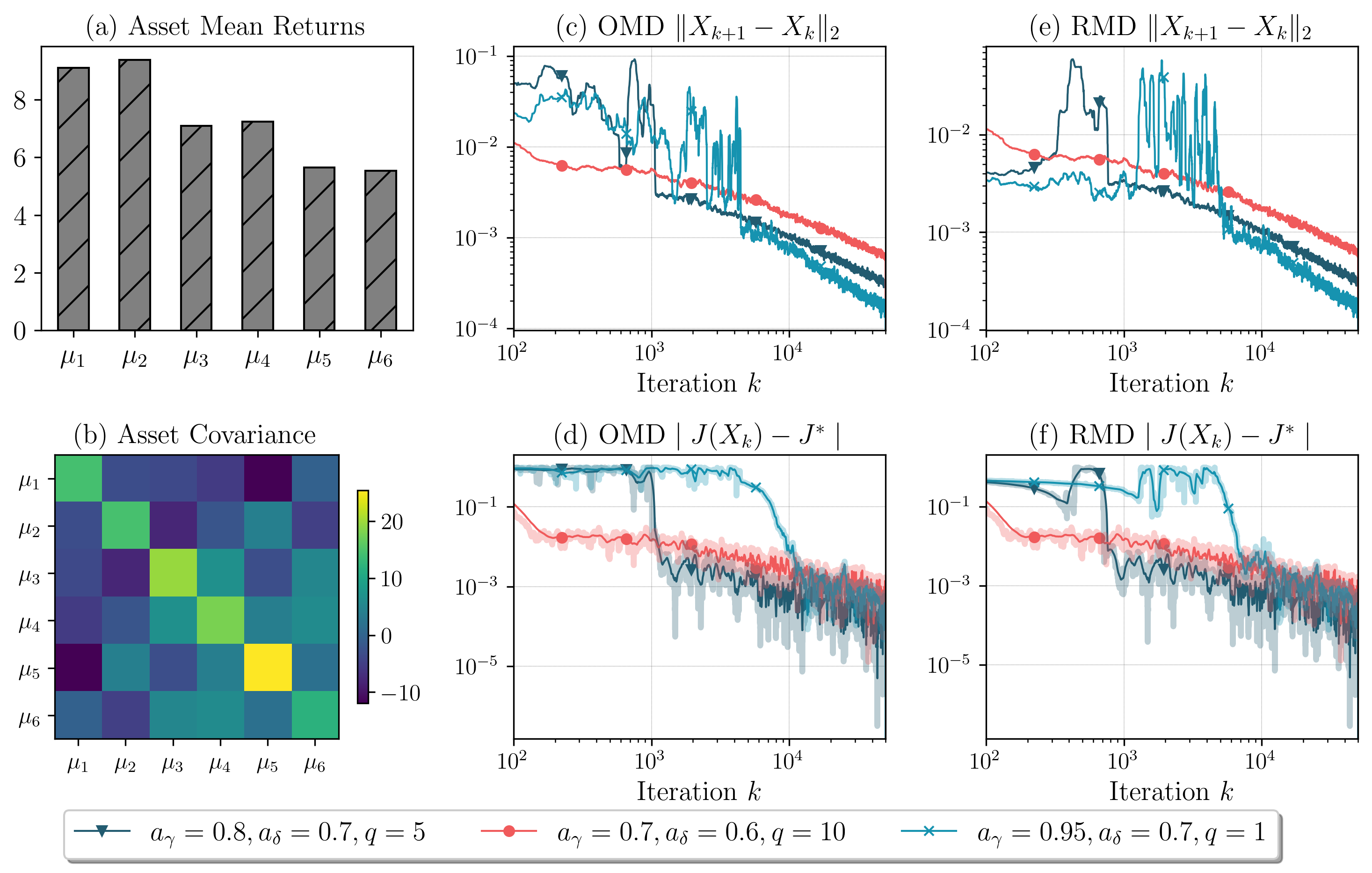}
    \caption{Performance of Algorithms~\ref{alg:optm-mrdesc} and \ref{alg:refl-desc} in Portfolio Optimization}
    \label{fig:portfolio}
\end{figure}


\subsection{Least Square Estimation in Linear Models}

Consider a dataset with data samples $\{(z_j, y_j)\}_{j=1}^{M}$ where $z_j \in \rset{N}{}$ represents an input vector and $y_j \in \rset{}{}$ denotes an output label. 
Moreover, $y_j$ and $z_j$ are related via a linear model, i.e., $y_j = w_0 + w^Tz_j + \xi_j$, where $w_0 \in \rset{}{}$, $w \in \rset{N}{}$, and $\xi_j$ represents some random noise. 
In this example, we assume that although the exact values of $w_0$ and $w$ are unavailable and need to be estimated from the data samples, $w_0$ and every entry of $w$ are known to be within some bounded interval $[-\bar{w}, \bar{w}]$ for some sufficiently large $\bar{w} \in \rset{}{++}$. 
For simplicity, let each $\Tilde{z}_j = [1; z_j]$, $\tilde{Z} = [\tilde{z}_1, \ldots, \tilde{z}_M]$, $y = [y_1; \cdots; y_M]$, and $\Tilde{w} = [w_0; w]$. 
Having all these in hand, the optimization problem can be written as:
\begin{align}
\underset{-\Bar{w} \leq \Tilde{w} \leq \Bar{w}}{\minimize} \; \frac{1}{2}\norm{\tilde{Z}^T\tilde{w} - y}^2_2
\end{align}
Similar to \cite[Sec.~VI]{gao2020continuous}, we proceed to recast the optimization problem above to a two-player zero-sum game. 
An auxiliary variable $\lambda \in \rset{M}{}$ is introduced such that we have the equivalence: 
$\frac{1}{2}\norm{\tilde{Z}^T\tilde{w} - y}^2_2 = \max_{\lambda \in \rset{M}{}} \lambda^T(\tilde{Z}^T\tilde{w} - y) - \frac{1}{2}\norm{\lambda}^2_2 = \max_{\lambda \in \rset{M}{}}J(\Tilde{w}, \lambda)$. 
Furthermore, the boundedness of $\Tilde{w}$ implies that of $\lambda$ and we also manually let $-\Bar{\lambda} \leq \lambda \leq \Bar{\lambda}$, for some $\Bar{\lambda}$ large enough. 
Denote the local objective functions $J^1(x^1; x^2) = J(x^1, x^2)$ and $J^2(x^2; x^1) = -J(x^1, x^2)$. 
Then, the two-player zero-sum game can be presented below:
\begin{align*}
\text{Player 1: }\underset{-\Bar{w} \leq x^1 \leq \Bar{w}}{\minimize}\; J^1(x^1; x^2), \; \text{Player 2: }\underset{-\Bar{\lambda} \leq x^2 \leq \Bar{\lambda}}{\minimize}\; J^2(x^2; x^1). 
\end{align*}
The associated pseudogradient is given by $F: 
\begin{psmallmatrix}x^1 \\ x^2\end{psmallmatrix} 
\mapsto M_{\text{lin}}
\begin{psmallmatrix}x^1 \\ x^2\end{psmallmatrix} 
+ \begin{psmallmatrix}0 \\ y\end{psmallmatrix}$, 
with the matrix $M_{\text{lin}} \coloneqq \begin{psmallmatrix} 0 & \Tilde{Z} \\ - \Tilde{Z}^T & I \end{psmallmatrix}$. 
The mere monotonicity of $F$ follows from the semi-positive definiteness of $M_{\text{lin}}$. 
To better visualize the problem, we consider the polynomial regression setting, where each $m$-th entry of the input vector $z_j$ is the $m$-th power of $[z_j]_1$. 
We select $N = 5$, $M = 10$, $\Bar{w} = \Bar{\lambda} = 5$, randomly sample each $[z_j]_1$ and $\xi_j$ from the compact intervals $[-1.5, 1.5]$ and $[-2, 2]$, respectively, and numerically confirm that $M_{\text{lin}}$ is full-rank such that the critical point $x^*$ of $F$ is unique in the interior. 

\begin{figure}
    \centering
    \includegraphics[width=\figurewidth\textwidth]{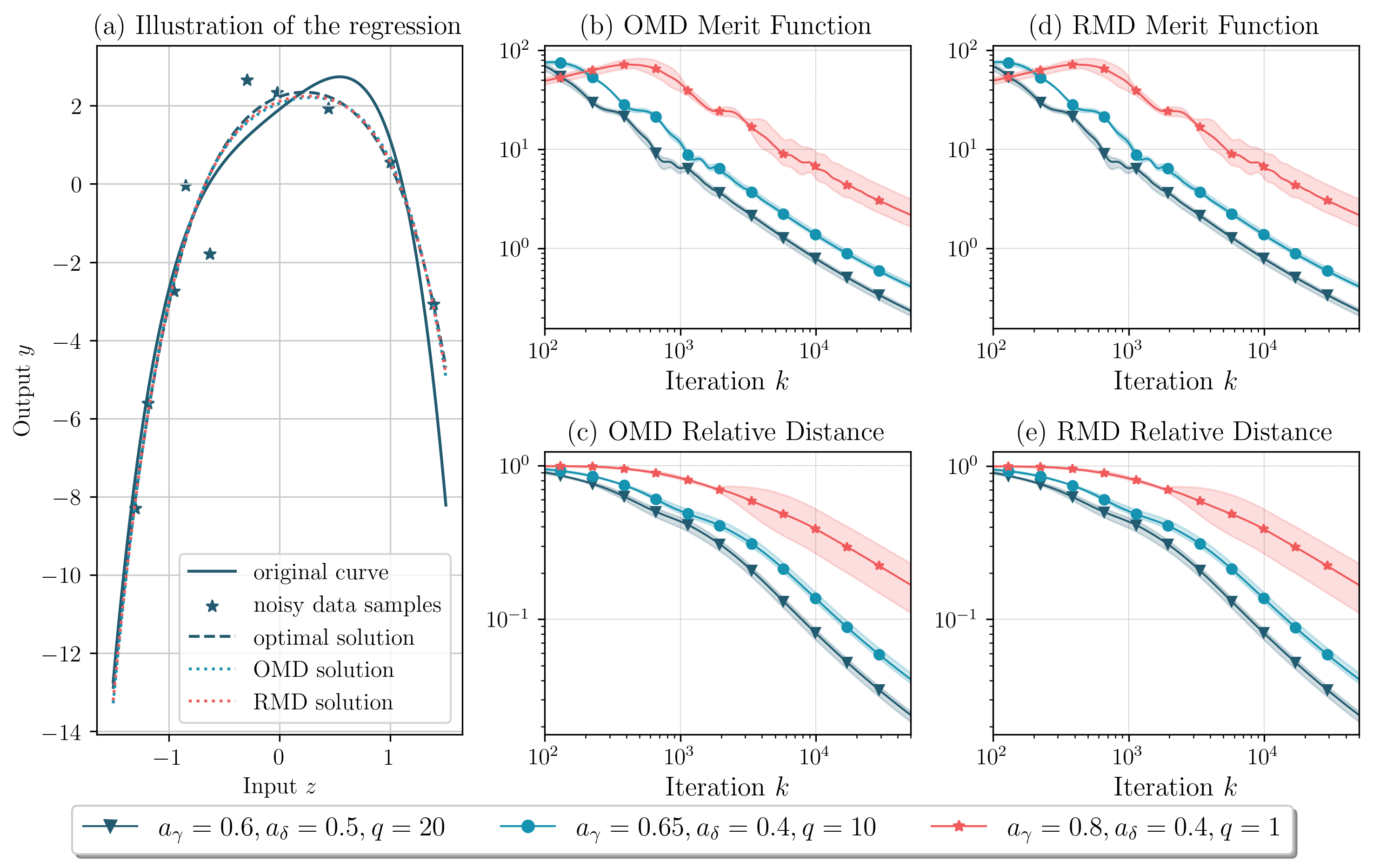}
    \caption{Performance of Algorithms~\ref{alg:optm-mrdesc} and \ref{alg:refl-desc} in LSE of Linear Models}
    \label{fig:linear-est}
\end{figure}

Set step size $\gamma_k = 1/(k+10^4)^{a_\gamma}$ and query radius $\delta_k = 1/(k+10^2)^{a_\delta}$. 
The experiments for OMD and RMD share the same random sample path. 
In Fig.~\ref{fig:linear-est}(a), we plot the original curve to fit, the noisy data samples used, and the two ergodic solutions obtained by OMD and RMD. 
One metric leveraged to measure the performance is the merit function \eqref{eq:merit-func}, and the results are visualized in Fig.~\ref{fig:linear-est}(b) and (d) under three different sets of parameters. 
We use solid lines to illustrate the average metric values of three different runs under the same choices of parameters and the semi-transparent envelope to indicate the true fluctuation. 
The relative distances to the unique optimal solution $x_*$, i.e., $\norm{\check{X}_k - x_*}_2 / \norm{x_*}_2$ are reported in Fig.~\ref{fig:linear-est}(c) and (e). 
We observe in the simulation that the convergence rates in Fig.~\ref{fig:linear-est}(b) and (d) match the results in Theorem~\ref{thm:erg-rate}. 
If the players only have a single observation per iteration, the decaying rate of step size should be increased and that of query radius decreased properly to tackle the estimation variance. 
In addition, we note that when OMD and RMD only differ in the looking-forward updating step, the ergodic results of one resemble those of the other. 

\subsection{Thermal Control in Buildings}\label{sect:thrm-ctrl}

In this example, we consider a load aggregator consisting of $N$ buildings, denoted by $\playerN \coloneqq \{1, \ldots, N\}$. 
Given an internal pricing mechanism defined in \cite{jiang2021game}, we would like to examine the convergence rates of the proposed bandit online NE learning algorithms in strongly pseudo-monotone problems. 
Given a time horizon $\mathcal{T} \coloneqq \{1, \ldots, T\}$, let $x^i_{t}$ denote the power consumption of building $i$ at time slot $t \in \mathcal{T}$, $x^i \coloneqq [x^i_{t}]_{t \in \mathcal{T}}$ the power profile of building $i$ over all the time slots, and $x \coloneqq [x^i]_{i \in \playerN}$ the concatenation of the energy profiles of all players. 

Assume that the aggregator purchases electricity from the wholesale energy market at an energy price of $p_e \in \rset{T}{++}$. 
In addition, there also exists an anytime demand charge rate of $p_d \in \rset{}{++}$ that penalizes the peak electricity usage of the aggregator during the time horizon under consideration. 
The peak electricity usage is characterized by the clique set of the participated buildings, which is denoted by $\mathcal{C} \coloneq \{\mathcal{C}_1, \ldots, \mathcal{C}_{n_c}\}$, where each $\mathcal{C}_j \subseteq \playerN$ for $j = 1, \ldots, n_c$. 
We use $y^i_t$ to denote the temperature of building $i$ during the $t$-th time slot, whose dynamics are described by the controlled LTI system $r^i_{t} = a^i r^i_{t-1} + b^i x^i_{t}$, $y^i_{t} = c^i r^i_{t}$. 
The desirable energy profile should strike a balance between controlling the indoor temperature within a comfortable zone $[\ubar{y}^i_{t}, \bar{y}^i_{t}]$ and reducing the energy cost. 
An additional constraint is that the air-conditioning power is upper-bounded by the system capacity for each building. 
Under this setting, given other buildings' power profile $x^{-i}$, each building aims at identifying an optimal power control strategy, which can be expressed as follows:
\begin{align}
\begin{split}
& \minimize_{x^i \in \mathcal{X}^i} \;(p_e)^Tx^i + Q^i(x^i) + p_d \cdot R^i(x)  \\
& \subj \;r^i_{t} = a^i r^i_{t-1} + b^i x^i_{t}, \forall t \in \mathcal{T} \\
& \qquad \qquad y^i_{t} = c^i r^i_{t}, \forall t \in \mathcal{T} \\
& \qquad \qquad  \ubar{y}^i_{t} \leq y^i_{t} \leq \bar{y}^i_{t}, \forall t \in \mathcal{T} \\
& \qquad \qquad  0 \leq x^i_{t} \leq \bar{x}^{i}, \forall t \in \mathcal{T}. 
\end{split}
\end{align}
Here, $Q^i$ is a strongly convex quadratic function added artificially for the purpose of convergence rates comparison, and $R^i$ denotes building $i$'s share of the aggregator peak demand. 
An approximate version of Shapley value is leveraged to distribute the collective demand charge among the players: 
\begin{align*}
& R^i(x) = \sum_{\mathcal{C}_j: i \in \mathcal{C}_j} \frac{(N - \abs{\mathcal{C}_j})!(\abs{\mathcal{C}_j} - 1)!}{N!}\Big(V(\mathcal{C}_j, x) - V(\mathcal{C}_j\backslash \{i\}, x)\Big), \\
& \text{where} \; V(\mathcal{C}_j, x) = \frac{1}{C}\log \Big( \sum_{t \in \mathcal{T}} \exp\big(\sum_{l \in \mathcal{C}_j} C x^l_{t}\big) \Big) \approx \max_{t \in \mathcal{T}} \Big\{\sum_{l \in \mathcal{C}_j} x^l_{t} \Big\}.
\end{align*}
In the above definition, $C$ is a manually chosen parameter that controls the accuracy of the approximation. 
We also note that the objective functions for this set of players admit a potential function given as follows: 
\begin{align*}
\Phi(x) & = \sum_{i \in \playerN} \Big((p_e)^Tx^i + Q^i(x^i)\Big) \\
& + p_d \cdot \sum_{\mathcal{C}_j \in \mathcal{C}} \frac{(N - \abs{\mathcal{C}_j})!(\abs{\mathcal{C}_j} - 1)!}{N!}V(\mathcal{C}_j, x). 
\end{align*}

Suppose that ten buildings $(N=10)$ participate in this game, and we conduct two sets of simulations where each building $i$ needs to decide its energy strategy regarding $2$ time slots $(T = 2)$ and $4$ time slots $(T = 4)$, respectively. 
The quadratic term $Q^i(x^i) = (x^i)^T\diag{\lambda_{i1}, \ldots, \lambda_{in^i}}x^i$ has each diagonal entry $\lambda_{ij}$ randomly sampled from $[0.04, 0.06]$. 
The metrics that we leverage to measure the performance of the methods are the relative distance between the NE and the perturbed actions, $\norm{\hat{X}_{k+1/2} - x_*}_2 / \norm{x_*}_2$, and the difference between the potential function's optimal value and the values at the perturbed actions, $\Phi(\hat{X}_{k+1/2}) - \Phi_*$. 
The simulation results for $T=2$ are reported in Fig.~\ref{fig:thrm-ctrl2} and those for $T=4$ in Fig.~\ref{fig:thrm-ctrl4}, where we compare the convergence speed of the proposed methods with the existing learning algorithms in \cite{bravo2018bandit,tatarenko2018learning,tatarenko2022rate,lin2021optimal}. 
The step size and query radius are set to be of the form $\gamma_k = \alpha_\gamma/(k + K_\gamma)^{a_\gamma}$ and $\delta_k = \alpha_\delta/(k + K_\delta)^{a_\delta}$, with the specific choices of the power parameters included in the legends of the figures. 
The parameters of the existing algorithms are selected based on the results in \cite[Thm.~5.2]{bravo2018bandit}, \cite[Thm.~3]{tatarenko2018learning}, \cite[Thm.~2]{tatarenko2022rate} and \cite[Thm.~2]{lin2021optimal}, respectively. 
The average of three runs with different random sample paths for each method is illustrated by the solid/dashed line, and the fluctuation is reflected through the semi-transparent envelope. 

From Fig.~\ref{fig:thrm-ctrl2} and \ref{fig:thrm-ctrl4}, it can be observed that Algorithms~\ref{alg:optm-mrdesc} and \ref{alg:refl-desc} can reduce the estimation variance and improve the convergence rates considerably. 
In Fig.~\ref{fig:thrm-ctrl2}, we consider two sets of parameters: Set~(a) that updates more aggressively $(a_{\gamma} = 0.95, a_{\delta} = 0.75)$ and Set~(b) more conservatively $(a_{\gamma} = 0.9, a_{\delta} = 0.6)$. 
For the first $2\times 10^4$ iterations, Set~(b) outperforms Set~(a), possibly due to that Set~(b) enjoys smaller constants $M_1$ and $M_2$ and starting iteration $K$ as stated in Theorem~\ref{thm:strmon-eucl-convgrate}. 
Nevertheless, as the algorithms proceed, Set~(a) outruns Set~(b), matching the decaying rate results in Theorem~\ref{thm:strmon-eucl-convgrate}. 
When we have $T$ increase from $2$ to $4$, the constants $M_1$, $M_2$, and $K$ grow significantly and it is advisable to choose a relatively conservative set of parameters to procure the desirable learning dynamics.

\begin{figure}
    \centering
    \includegraphics[width=\figurewidth\textwidth]{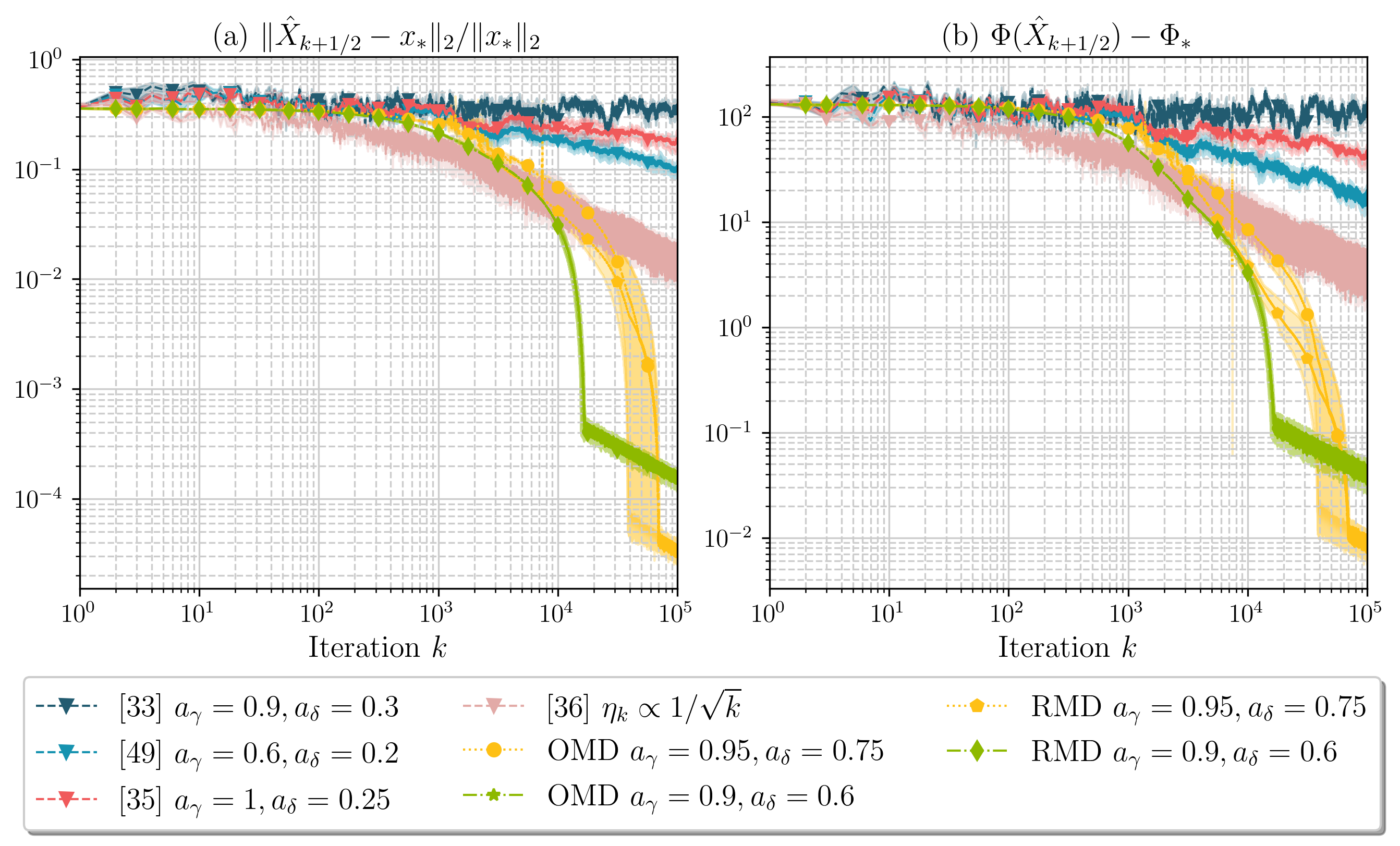}
    \caption{Performance of Algorithms~\ref{alg:optm-mrdesc} and \ref{alg:refl-desc} in Thermal Control Problem ($T=2$)}
    \label{fig:thrm-ctrl2}
\end{figure}

\begin{figure}
    \centering
    \includegraphics[width=\figurewidth\textwidth]{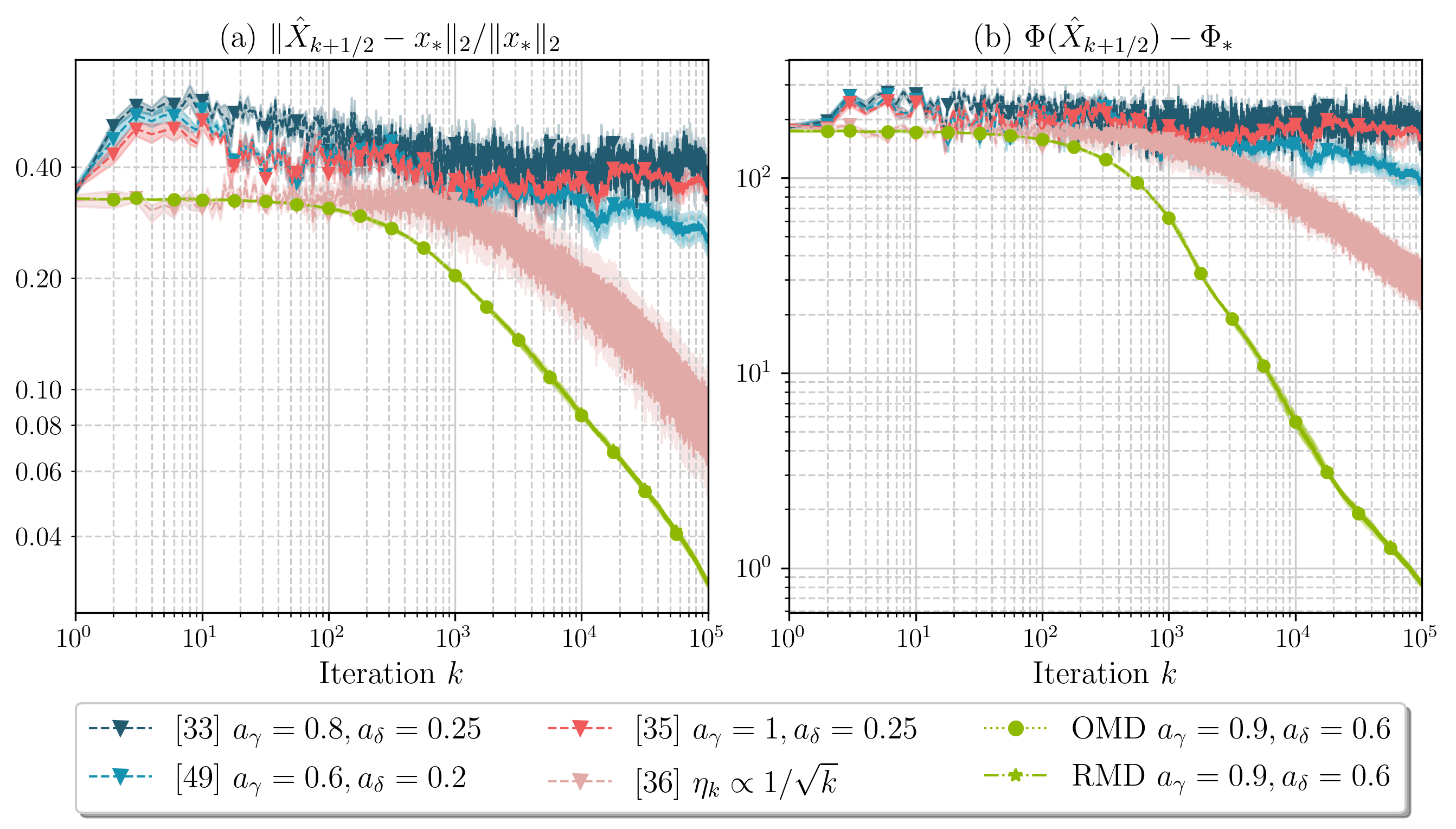}
    \caption{Performance of Algorithms~\ref{alg:optm-mrdesc} and \ref{alg:refl-desc} in Thermal Control Problem ($T=4$)}
    \label{fig:thrm-ctrl4}
\end{figure}

\section{Conclusion and Future Directions}\label{sect:conclu}
In this paper, we study bandit learning in multi-player continuous games and propose two learning algorithms that are constructed by combining the residual pseudogradient estimation and two single-call extra-gradient schemes, i.e., optimistic mirror descent and reflected mirror descent, respectively. 
The actual sequences of play of the proposed algorithms are proven to converge to a critical point of a pseudo-monotone plus games a.s. 
Furthermore, in strongly pseudo-monotone games, the proposed algorithms can achieve an optimal convergence rate of $O(1/t^{1-\epsilon})$, which dramatically ameliorates the convergence speed as learning schemes with a single oracle call. 
There remain several open problems. 
In the problem formulation, we implicitly assume that the realized objective function values are accurate, and it remains an open question to answer how to achieve the same or comparable convergence rate when they are corrupted by random noise. 
Another potential future direction resides in extending the a.s. convergence results for the actual sequence of play to more general classes of games such as merely monotone games without resorting to Tikhonov regularization. 
We intend to address these questions in future work. 

\appendices 
\section*{Appendix}
\addcontentsline{toc}{section}{Appendix}
\renewcommand{\thesubsection}{\Alph{subsection}}

\newtheorem{appdxlemma}{Lemma}
\numberwithin{appdxlemma}{subsection} 
\numberwithin{equation}{subsection} 

\subsection{Preamble}\label{appd:preamble}

\begin{appdxlemma}\label{le:distgen-func}
Consider the ambient Banach space $\mathcal{B}$ equipped with norm $\norm{\cdot}$ and a closed and convex feasible set $\mathcal{X} \subseteq \cl{\dom \psi} \subseteq \mathcal{B}$.
Suppose $\psi: \dom \psi \to \rset{}{}$ is a DGF, then for all $x \in \dom \psi \cap \mathcal{X}$ and $y \in \mathcal{B}^*$, the following relations hold : 
\begin{outline}[enumerate]
\1 $x = \nabla \psi^* (y) \iff y \in \nabla \psi(x) + N_{\mathcal{X}}(x)$;  
\1 The following are equivalent: $x^+ = P_{x, \mathcal{X}}(y) \iff \nabla \psi (x) + y \in \nabla \psi(x^+) + N_{ \mathcal{X}}(x) \iff x^+ = \nabla \psi^*(\nabla \psi (x) + y)$; 
\1 $x = \nabla \psi^* (y)$ and $p \in \mathcal{X} \implies \langle \nabla \psi(x), x - p\rangle \leq \langle y, x-p\rangle$;
\1 The mirror map $\nabla \psi^*$ and the prox-mapping $P_{x, \mathcal{X}}$ are $1/\tilde{\mu}$-Lipschitz continuous, i.e., for any $y_1$ and $y_2 \in \mathcal{B}^*$, 
$\norm{\nabla \psi^*(y_1) - \nabla \psi^*(y_2)} \leq \frac{1}{\tilde{\mu}}\norm{y_1 - y_2}_*$ and 
$\norm{P_{x, \mathcal{X}}(y_1) - P_{x, \mathcal{X}}(y_2)} \leq \frac{1}{\tilde{\mu}}\norm{y_1 - y_2}_*$; 
\1 $\nabla \psi^* \circ \nabla \psi: \dom \psi \cap \mathcal{X} \to \dom \psi \cap \mathcal{X}$ is an identity map. 
\end{outline}
\end{appdxlemma}
\begin{proof}
$(\romannum{1})$ The converse directly follows from the first-order optimality conditions for constrained optimization, i.e., if the zero inclusion $0 \in \nabla \psi(x) - y + N_{\mathcal{X}}(x)$ holds for some $x \in \dom \psi \cap \mathcal{X}$ and $y \in \mathcal{B}^*$, then $x = \argmax_{x \in \mathcal{X}}\{\langle y, x\rangle - \psi(x)\}$. 
Regarding the other direction, the emphasis is on arguing the fact that we can restrict our discussion of the solution from $\mathcal{X} = \cl{\dom\psi} \cap \mathcal{X}$ to $\dom\psi \cap \mathcal{X}$. 
Given an arbitrary $y \in \mathcal{B}^*$, a solution $x$ to the maximization problem in $\nabla \psi^*(y)$ satisfies $0 \in \nabla \psi(x) - y + N_{\mathcal{X}}(x)$. 
Since $\nabla \psi(\dom \psi) = \rset{n}{}$ and $N_{\mathcal{X}}(x) \subseteq \rset{n}{}$ for all $x \in \mathcal{X}$, there always exists an $x^\star \in \dom \psi \cap \mathcal{X}$ to make the zero inclusion above hold. 
Besides, the strong concavity of $-\psi$ indicates that $\nabla \psi^*(y)$ admits the unique solution $x^\star$. 

$(\romannum{2})$ The equivalence is also a straightforward result of the first-order optimality condition discussed in $(\romannum{1})$, and for the same reason, we can claim $x^+ \in \dom \psi \cap \mathcal{X}$. 

$(\romannum{3})$ By noting that $y - \nabla \psi(x) \in N_{\mathcal{X}}(x)$ and the definition of the normal cone, i.e., $\phi \in N_{\mathcal{X}}(x)$ if $\langle \phi, p - x\rangle \leq 0$ for all $p \in \mathcal{X}$, we can reach the relation that $\langle y - \nabla \psi(x), p - x\rangle \leq 0$. 

$(\romannum{4})$ It suffices for us to prove the Lipschitz continuity of $\nabla \psi^*$ and that of $P_{x, \mathcal{X}}(y)$ directly follows from the equivalence given in $(\romannum{2})$.
For $y_1, y_2 \in \mathcal{B}^*$ arbitrary, let $x_1 = \nabla \psi^*(y_1)$ and $x_2 = \nabla \psi^*(y_2)$, and $(\romannum{3})$ gives that 
$\langle \nabla\psi(x_1), x_1 - x_2\rangle \leq \langle y_1, x_1 - x_2 \rangle$ 
and $\langle \nabla\psi(x_2), x_2 - x_1\rangle \leq \langle y_2, x_2 - x_1 \rangle$. 
Combining both yields $\langle \nabla\psi(x_1) - \nabla\psi(x_2), x_1 - x_2 \rangle \leq \langle y_1 - y_2 , x_1 - x_2\rangle$. 
By the Cauchy-Schwarz inequality, $\langle y_1 - y_2 , x_1 - x_2\rangle \leq \norm{y_1 - y_2}_*\cdot\norm{x_1 - x_2}$. 
The conclusion then readily follows from the $\tilde{\mu}$-strongly convexity of $\psi$, i.e., $\langle \nabla\psi(x_1) - \nabla\psi(x_2), x_1 - x_2 \rangle \geq \tilde{\mu} \norm{x_1 - x_2}^2$.

$(\romannum{5})$ Let $x^+ = \nabla \psi^*(\nabla \psi(x))$. Apparently, $\nabla \psi(x) \in \nabla \psi(x^+) + N_{\mathcal{X}}(x^+)$ and $0 \in N_{\mathcal{X}}(x^+)$, which together imply $x^+ = x$ and it is uniquely determined. 
\end{proof}

\begin{appdxlemma}\label{le:stdeq}
Suppose that $\psi$ is a $\tilde{\mu}$-strongly convex DGF on $\mathcal{S}_2 \subseteq \mathcal{S}_1 \subseteq \cl{\dom{\psi}}$ and $D$ its associated Bregman divergence. 
For any $p \in \mathcal{S}_2$ and $x \in \dom{\psi}$, let $x^+_1 = P_{x, \mathcal{S}_1}(y_1)$ and $x^+_2 = P_{x, \mathcal{S}_2}(y_2)$. 
Then the following inequality holds
\begin{align}
\begin{split}\label{eq:stdeq-1}
& D(p, x^+_2) \leq D(p, x) + \langle y_2, x^+_1 - p\rangle + \langle y_1 - y_2, x^+_1 - x^+_2\rangle \\
& \qquad\qquad - D(x^+_2, x^+_1) - D(x^+_1, x) 
\end{split} \\
& \leq D(p, x) + \langle y_2, x^+_1 - p\rangle + \frac{1}{2\tilde{\mu}}\norm{y_2 - y_1}^2_* - \frac{\tilde{\mu}}{2}\norm{x^+_1 - x}^2. \label{eq:stdeq-2}
\end{align}
\begin{proof}
By the "three-point identity" of the Bregman divergence, we can relate $x^+_2$ to $x$ as follows: 
\begin{align*}
D(p, x^+_2) &= D(p, x) - D(x^+_2, x) + \langle \nabla \psi(x^+_2) - \nabla \psi(x), x^+_2 - p\rangle \\
&\overset{(a)}{\leq} D(p, x) - D(x^+_2, x) + \langle y_2, x^+_2 - p \rangle, 
\end{align*}
where in $(a)$, we use the fact that $x^+_2 = \nabla \psi^*(\nabla \psi(x) + y_2)$ on the set $\mathcal{S}_2$ and $p \in \mathcal{S}_2$, which imply $\langle \nabla \psi(x^+_2), x^+_2 - p \rangle \leq \langle \nabla \psi(x) + y_2, x^+_2 - p\rangle$ by Lemma~\ref{le:distgen-func}$(\romannum{3})$. 
Since $x^+_2 \in \mathcal{S}_2 \subseteq \mathcal{S}_1$, again by the similar arguments as above, we have:
\begin{align*}
D(x^+_2, x^+_1) \leq D(x^+_2, x) - D(x^+_1, x) + \langle y_1, x^+_1 - x^+_2\rangle.
\end{align*}
Combining the two inequalities above gives \eqref{eq:stdeq-1}. 
By the lower bound of $D$ inherited from the $\tilde{\mu}$-strong monotonicity of $\psi$, we get $D(x^+_2, x^+_1) \geq \frac{\tilde{\mu}}{2}\norm{x^+_2 - x^+_1}$ and $D(x^+_1, x) \geq \frac{\tilde{\mu}}{2}\norm{x^+_1 - x}^2$. 
In addition, by the Cauchy-Schwarz inequality, $\langle y_1 - y_2, x^+_1 - x^+_2\rangle \leq \frac{1}{2\tilde{\mu}}\norm{y_1 - y_2}^2 + \frac{\tilde{\mu}}{2}\norm{x^+_1 - x^+_2}$. 
Taking all the above into account, we get the desired bound in \eqref{eq:stdeq-2}. 
\end{proof}
\end{appdxlemma}

\begin{appdxlemma}\label{le:seq-convg}
Consider two non-negative sequences $(\gamma_k)_{k\in \nset{}{+}} \in \rset{\nset{}{}}{}$ and $(\delta_k)_{k \in \nset{}{+}}\in \rset{\nset{}{}}{}$.
Suppose $\lim_{k \to \infty} \gamma_k = 0$ and $\lim_{k \to \infty} \delta_k = 0$. 
In addition, let $\sum_{k \in \nset{}{}} \gamma_k = +\infty$ and $\lim_{k\to\infty} \gamma_k/\delta_k = 0$. 
Then, for any $K \in \nset{}{+}$ fixed, we have $\lim_{k\to\infty}\gamma_{k - K}/\delta_k = 0$. 
\end{appdxlemma}
\begin{proof}
To study $\gamma_{k-K}/\delta_k$, we note that $\gamma_k/\delta_k + (\gamma_{k-K}-\gamma_k)/\delta_k$, and it suffices to show that the second part converges to zero. 
To do so, we consider $\sum_{k=K}^{\infty}(\gamma_{k-K}-\gamma_k) = \lim_{n\to\infty} \sum_{k=K}^{n}(\gamma_{k-K}-\gamma_k) = \lim_{n\to\infty} (\sum_{k=0}^{K-1} \gamma_k - \sum_{k=n-K+1}^{n}\gamma_k)$ by telescoping. 
Moreover, $\sum_{k=0}^{K-1} \gamma_k$ is merely a constant and $\lim_{n\to\infty}\sum_{k=n-K+1}^{n}\gamma_k = 0$. 
Thus $\sum_{k=K}^{\infty}(\gamma_{k-K}-\gamma_k) < \infty$ and $(\gamma_{k}-\gamma_{k+K})_{k \in \nset{}{}}$ decays faster than $(\gamma_{k})_{k \in \nset{}{}}$ in the tail, from which our claim follows. 
\end{proof}

\subsection{Results in Residual Pseudogradient Estimate}\label{appd:rpg}

\begin{proof}
(Proof of Lemma~\ref{le:unbiased})    
We first note that the R.H.S. can be simplified as follows. 
\begin{align*}
& \expt{}{G^i_k \mid \mathcal{F}_k} = \frac{n^i}{\delta_k}\Bexpt{}{\Big(J^i(\hat{X}_{k+1/2}) - J^i(\hat{X}_{k-1/2})\Big)u^i_k \mid \mathcal{F}_k} \\
& = \frac{n^i}{\delta_k}\expt{}{J^i(\hat{X}_{k+1/2})u^i_k \mid \mathcal{F}_k} - \frac{n^i}{\delta_k}J^i(\hat{X}_{k-1/2})\expt{}{u^i_k} \\
& = \frac{n^i}{\delta_k}\expt{}{J^i(\hat{X}_{k+1/2})u^i_k \mid \mathcal{F}_k}. 
\end{align*}
Next, we are going to prove that $\nabla_{x^i} \Tilde{J}^i_{\delta_k}(\Bar{X}_{k+1/2})$ is a version of $\frac{n^i}{\delta_k}\expt{}{J^i(\hat{X}_{k+1/2})u^i_k \mid \mathcal{F}_k}$. 
The complete expression of $\nabla_{x^i} \Tilde{J}^i_{\delta_k}(\Bar{X}_{k+1/2})$ is given below: 
\begin{align*}
\frac{1}{\mathbb{V}^i}\nabla_{x^i} \int_{\delta_k \mathbb{S}_{-i}}\int_{\delta_k \mathbb{B}_i} J^i(\Bar{X}^i_{k+1/2} + \Tilde{\tau}^i; \Bar{X}^{-i}_{k+1/2} + \tau^{-i})d\Tilde{\tau}^i d\tau^{-i}. 
\end{align*}
To examine the interchangeability of integration and derivative, we check the equation above entry-wisely, i.e., we start by analyzing $\partial_{[x^i]_l}$ for an $l \in \{1, \ldots, n^i\}$ and proceed with other entries similarly. 
Since $J^i$ is assumed to be differentiable in $x^i$, by the mean value theorem, for arbitrary $x \in \mathcal{X}$, 
\begin{align*}
    \frac{1}{\varepsilon} \big(J^i(x + \varepsilon e_l) - J^i(x)\big) =  \partial_{[x^i]_l} J^i(x + \alpha \cdot \varepsilon e_l), 
\end{align*}
where $\epsilon$ denotes a small coefficient; 
$\alpha$ is a proper constant on $[0, 1]$. 
Under the current regularity setup, the partial derivative $\partial_{[x^i]_l}J^i$ is locally bounded with bounds dependent on the point $\bar{X}_{k+1/2}$ and its neighbor we are examining, which further implies 
\begin{align*}
\underset{\delta_k \mathbb{S}_{-i}}{\int}\underset{\delta_k \mathbb{B}_i}{\int} \mid\partial_{[x^i]_l} J^i(\Bar{X}^i_{k+1/2}(\omega) + \Tilde{\tau}^i; \Bar{X}^{-i}_{k+1/2}(\omega) + \tau^{-i})\mid d\Tilde{\tau}^i d\tau^{-i}
< \infty
\end{align*}
for arbitrary possible random sample path $\omega \in \Omega$. 
By the dominated convergence theorem, we can exchange the limit of passing $\varepsilon$ to $0$ and the integration over $\delta_k \mathbb{S}_{-i} \times \delta_k \mathbb{B}_i$ and obtain:
\begin{align*}
& \nabla_{x^i} \Tilde{J}^i_{\delta_k}(\Bar{X}_{k+1/2}) \\
& = \frac{1}{\mathbb{V}^i} \int_{\delta_k \mathbb{S}_{-i}}\int_{\delta_k \mathbb{B}_i} \nabla_{x^i}J^i(\Bar{X}^i_{k+1/2} + \Tilde{\tau}^i; \Bar{X}^{-i}_{k+1/2} + \tau^{-i})d\Tilde{\tau}^i d\tau^{-i} \\
& \overset{(a)}{=} \frac{1}{\mathbb{V}^i} \int_{\delta_k \mathbb{S}_{\mathcal{N}}} J^i(\Bar{X}^i_{k+1/2} + \tau^i; \Bar{X}^{-i}_{k+1/2} + \tau^{-i})\frac{\tau^i}{\norm{\tau^i}} d\tau \\
& \overset{(b)}{=} \frac{n^i/\delta_k}{\vol{\delta_k \mathbb{S}_{\mathcal{N}}}} \int_{\delta_k \mathbb{S}_{\mathcal{N}}}  J^i(\Bar{X}^i_{k+1/2} + \tau^i; \Bar{X}^{-i}_{k+1/2} + \tau^{-i}) \frac{\tau^i}{\delta_k} d\tau\\
& \overset{(c)}{=} \frac{n^i/\delta_k}{\vol{ \mathbb{S}_{\mathcal{N}}}} \int_{\mathbb{S}_{\mathcal{N}}} J^i(\bar{X}_{k+1/2} + \delta_k \tau)\tau^i d{\tau}, 
\end{align*}
where in $(a)$, we let $\mathbb{S}_{\mathcal{N}} \coloneqq \prod_{i \in \mathcal{N}} \mathbb{S}_i$, and the equality follows from the Stoke's theorem \cite[Thm.~9.3.1]{galbis2012vector} since $J^i$ is $C^1$ in $x^i$; 
$(b)$ is a direct results of the fact that $\vol{\delta_k \mathbb{B}_i} = (\delta_k/n_i)\vol{\delta_k \mathbb{S}_i}$; 
for $(c)$, we simply apply the change of variables.
For any event $E \in \mathcal{F}_k$, we can check:
\begin{align*}
& \expt{}{\nabla_{x^i} \Tilde{J}^i_{\delta_k}(\Bar{X}_{k+1/2}) \mathds{1}_E} \\
& = \frac{n^i/\delta_k}{\vol{\mathbb{S}_{\mathcal{N}}}} \cdot \int_{E}\int_{\mathbb{S}_{\mathcal{N}}}  J^i(\Bar{X}_{k+1/2}(\omega) + \delta_k\tau) \tau^i d\tau \mathcal{P}(d\omega) \\
& \overset{(a)}{=} \frac{n^i}{\delta_k} \int_{\Bar{X}^i_{k+1/2}(E)}\int_{\mathbb{S}_{\mathcal{N}}} J^i(x + \delta_k\tau)\tau^i \tilde{\mu}_k(d\tau) \cdot \nu_{k+1/2}(dx)\\
& \overset{(b)}{=} \frac{n^i}{\delta_k} \int_{\Bar{X}^i_{k+1/2}(E) \times \mathbb{S}_{\mathcal{N}}} J^i(x + \delta_k\tau)\tau^i (\tilde{\mu}_k \times \nu_{k+1/2})(dx, d\tau) \\
& \overset{(c)}{=} \frac{n^i}{\delta_k} \int_{E \times \Omega}J^i(\Bar{X}_{k+1/2}(\omega_1) + \delta_ku_k(\omega_2)) u^i_k(\omega_2)(\mathcal{P} \times \mathcal{P})(d\omega_1, d\omega_2) \\
& \overset{(d)}{=} \frac{n^i}{\delta_k}\expt{}{J^i(\bar{X}_{k+1/2} + \delta_ku_k)u^i_k \mathds{1}_E},
\end{align*}
where in $(a)$, we let $\tilde{\mu}_k \coloneqq \mathcal{P} \circ (u_k)^{-1}$ denote the probability measure of $u_k$, which corresponds to the uniform distribution over $\mathbb{S}_{\mathcal{N}}$ and $\nu_{k+1/2}$ the probability measure of $\bar{X}_{k+1/2}$, 
and the equality in $(a)$ follows from the change of variables formula for computing the expected value of $\bar{X}^i_{k+1/2}$. 
To have the relation in $(b)$, we first note that $\Bar{X}^i_{k+1/2}(E) 
 + \delta_k\mathbb{S}_{\mathcal{N}}$ is a subset of $\mathcal{X}_a$, which is bounded by assumption and construction. 
Hence, we have the integrability $\int_{\Bar{X}^i_{k+1/2}(E) 
 \times \mathbb{S}_{\mathcal{N}}} \mid J^i(x + \tau) \mid (\tilde{\mu}_k \times \nu_{k+1/2})(dx, d\tau) \leq \max_{x \in \mathcal{X}} J^i(x) < \infty$, and the Fubini's theorem can be applied here to obtain the equality. 
For $(c)$, we again use the change of variables formula for computing the expected value involving $\bar{X}^i_{k+1/2}$ and $u_k$.
Finally, $(d)$ holds as a result of the fact that $u_k$ is independent of $\bar{X}^i_{k+1/2}$ and $\mathds{1}_E$. 
This finishes the proof that $\nabla_{x^i} \tilde{J}^i_{\delta_k}(\bar{X}_{k+1/2})$ is a version of the conditional expectation $\expt{}{G^i_k \mid \mathcal{F}_k}$ by definition \cite[Sec.~4.1]{durrett2019probability}. 
\end{proof}

\begin{proof}(Proof of Lemma~\ref{le:codom-err})

The systematic error can be separated as follows: 
\begin{align*}
& \norm{B^i_k}_* = \norm{\nabla_{x^i}\tilde{J}^i_{\delta_k}(\bar{X}_{k+1/2}) - \nabla_{x^i}J^i(X_{k+1/2})}_* \leq \\
& \underbrace{\norm{\nabla_{x^i}\tilde{J}^i_{\delta_k}(\bar{X}_{k+\frac{1}{2}}) - \nabla_{x^i}J^i(\bar{X}_{k+\frac{1}{2}})}_*}_{(\romannum{1})} + \underbrace{\norm{\nabla_{x^i}J^i(\bar{X}_{k+\frac{1}{2}})- \nabla_{x^i}J^i(X_{k+\frac{1}{2}})}_*}_{(\romannum{2})}. 
\end{align*}
For $(\romannum{1})$, we start by applying the same arguments in the proof of Lemma~\ref{le:unbiased} to switch the derivative $\nabla_{x^i}$ and the integration for the computation of $\nabla_{x^i}\tilde{J}^i_{\delta_k}(\bar{X}_{k+\frac{1}{2}})$, which gives
\begin{align*}
& (\romannum{1}) = \Bnorm{\frac{1}{\mathbb{V}^i} \int_{\delta_k \mathbb{S}_{-i}}\int_{\delta_k \mathbb{B}_i} \big(\nabla_{x^i}J^i(\Bar{X}^i_{k+1/2} + \Tilde{\tau}^i; \Bar{X}^{-i}_{k+1/2} + \tau^{-i})\\
& \qquad\qquad - \nabla_{x^i}J^i(\bar{X}^i_{k+1/2}; \bar{X}^{-i}_{k+1/2}) \big) d\Tilde{\tau}^i d\tau^{-i}}_* \\
& \overset{(a)}{\leq} \frac{1}{\mathbb{V}^i}\int_{\delta_k \mathbb{S}_{-i}}\int_{\delta_k \mathbb{B}_i}
\bnorm{\nabla_{x^i}J^i(\Bar{X}^i_{k+1/2} + \Tilde{\tau}^i; \Bar{X}^{-i}_{k+1/2} + \tau^{-i}) \\
& \qquad\qquad - \nabla_{x^i}J^i(\bar{X}^i_{k+1/2}; \bar{X}^{-i}_{k+1/2})}_*
d\Tilde{\tau}^i d\tau^{-i} \\
& \overset{(b)}{\leq} \frac{L^i}{\mathbb{V}^i}\int_{\delta_k \mathbb{S}_{-i}}\int_{\delta_k \mathbb{B}_i} \norm{[\tilde{\tau}^i; \tau^{-i}]} d\Tilde{\tau}^i d\tau^{-i} \\
& \overset{(c)}{\leq} L^i\bar{u}_{\mathcal{N}} \cdot \delta_k, 
\end{align*}
where $(a)$ is immediate from the triangle inequality; 
in $(b)$, we use the $L^i$-Lipschitz continuity of $\nabla_{x^i}J^i$; 
finally for $(c)$, we substitute the integrand with the constant $\bar{u}_{\mathcal{N}}$, where $\bar{u}_{\mathcal{N}} \coloneqq \norm{u}$ for some $u \in \mathbb{S}_{\mathcal{N}} \coloneqq \prod_{i \in \mathcal{N}}\mathbb{S}_{i}$. 

An upper bound for $(\romannum{2})$ can be trivially constructed by again using the Lipschitz continuity of $\nabla_{x^i}J^i$ as follows:
\begin{align*}
(\romannum{2}) \leq L^i \cdot \bnorm{\frac{\delta_k}{r^i}(X_{k+1/2} - p)} \leq \frac{L^i D_{\mathcal{X}_a}}{r^i}\delta_k, 
\end{align*}
where $D_{\mathcal{X}_a}$ denotes the diameter of the global action space $\mathcal{X}_a$ measured by the given norm $\norm{\cdot}$. 
The stack of systematic error $\norm{B_k}_*$ then enjoys the bound $\norm{B_k}_* \leq \alpha_B \delta_k$ where $\alpha_B = \sum_{i \in \mathcal{N}}L^i(\bar{u}_{\mathcal{N}}+D_{\mathcal{X}_a}/r^i)$. 

For the stochastic error $V_k$, by the triangle inequality, we have
\begin{align*}
\norm{V^i_k}_* \leq \underbrace{\norm{G^i_k}_*}_{(\romannum{1})} + \underbrace{\norm{\nabla_{x^i} \tilde{J}^i_{\delta_k}(\bar{X}_{k+1/2})}_*}_{(\romannum{2})}. 
\end{align*}
For $(\romannum{1})$, by noting that $J^i$ is continuous on the compact set $\mathcal{X}_a$ and its maximum exists, we obtain $
\norm{G^i_k}_* \leq \frac{n^i}{\delta_k} \cdot 2\max_{x \in \mathcal{X}_a} \mid J^i(x) \mid \cdot \norm{u^i_k}_*$, where $\norm{u^i_k}_*$ is the $\norm{\cdot}_*$ norm of the unit vector sampled from $\mathbb{S}^{n_i}$. 
We let $\bar{u}^i_* \coloneqq \norm{u^i_k}_*$. 
For $(\romannum{2})$, we again leverage the arguments in the proof of Lemma~\ref{le:unbiased} to interchange the partial derivative and the integral and get
\begin{align*}
& \norm{\nabla_{x^i} \tilde{J}^i_{\delta_k}(\bar{X}_{k+1/2})}_* = \frac{n^i/\delta_k}{\vol{ \mathbb{S}_{\mathcal{N}}}}\norm{\int_{\mathbb{S}_{\mathcal{N}}} J^i(\bar{X}_{k+1/2} + \delta_k \tau){\tau}^i d\tau}_* \\
& \leq \frac{n^i/\delta_k}{\vol{ \mathbb{S}_{\mathcal{N}}}} 
\int_{\mathbb{S}_{\mathcal{N}}} \mid J^i(\bar{X}_{k+1/2} + \delta_k \tau) \mid \norm{{\tau}^i}_* d\tau \\
& \leq \frac{n^i}{\delta_k} \cdot \max_{x \in \mathcal{X}_a} \mid J^i(x) \mid \cdot \bar{u}^i_*.
\end{align*}
Altogether, we have that the stack stochastic error vector $\norm{V_k}_* \leq \alpha_V/\delta_k$, where $\alpha_V = \sum_{i \in \mathcal{N}}3n^i\max_{x \in \mathcal{X}_a} \mid J^i(x) \mid \bar{u}^i_*$.
\end{proof}

\begin{proof}
(Proof of Lemma~\ref{le:bdd-stoch-err}) \\
Similarly, we first separate the squared norm, i.e., 
\begin{align*}
\norm{V_k}^2_* \leq 2\norm{G^i_k}^2_* + 2\norm{\nabla_{x^i} \tilde{J}^i_{\delta_k}(\bar{X}_{k+1/2})}^2_*. 
\end{align*}
For the second part, a constant upper bound can be procured: 
\begin{align*}
& \norm{\nabla_{x^i} \Tilde{J}^i_{\delta_k}(\Bar{X}_{k+1/2})}^2_* \\
& \overset{(a)}{\leq} \frac{1}{\mathbb{V}^i} \int_{\delta_k \mathbb{S}_{-i}}\int_{\delta_k \mathbb{B}_i} \norm{\nabla_{x^i}J^i(\Bar{X}^i_{k+1/2} + \Tilde{\tau}^i; \Bar{X}^{-i}_{k+1/2} + \tau^{-i})}^2_* d\Tilde{\tau}^i d\tau^{-i} \\
& \overset{(b)}{\leq} \big(\max_{x \in \mathcal{X}_a} \norm{\nabla_{x^i}J^i(x)}^2_*\big) \cdot \frac{1}{\mathbb{V}^i}\int_{\delta_k \mathbb{S}_{-i}}\int_{\delta_k \mathbb{B}_i} 1 d\Tilde{\tau}^i d\tau^{-i} \\
& = \max_{x \in \mathcal{X}_a} \norm{\nabla_{x^i}J^i(x)}^2_*
\end{align*}
where we obtain $(a)$ by applying the triangle and Jensen's inequality; 
since $\norm{\nabla_{x^i}J^i(\cdot)}^2_*$ by assumption is a continuous function on $\mathcal{X}_a$, the maximum exists in $(b)$ and hence it admits a constant upper bound. 
For the first part, we leverage the differentiability of $J^i$ in $x$ and the mean value theorem to obtain
\begin{align*}
& \norm{G^i_k}^2_* \overset{(a)}{=} \big(\frac{n^i}{\delta_k}\big)^2\Big( \langle \nabla_{x}J^i(\Tilde{X}), \hat{X}_{k+1/2} - \hat{X}_{k-1/2}\rangle \Big)^2\norm{u^i_k}^2_* \\
& \overset{(b)}{\leq} \big(\frac{n^i}{\delta_k}\big)^2 \norm{\nabla_{x}J^i(\tilde{X})}^2_* \cdot \norm{\hat{X}_{k+1/2} - \hat{X}_{k-1/2}}^2 \cdot \norm{u^i_k}^2_* \\
& \overset{(c)}{\leq} \big(\frac{n^i}{\delta_k}\big)^2\bar{\nabla}^2_i(\bar{u}^i_*)^2 \cdot \norm{\hat{X}_{k+1/2} - \hat{X}_{k-1/2}}^2. 
\end{align*}
In $(a)$, $\tilde{X}$ is some convex combination of $\hat{X}_{k+1/2}$ and $\hat{X}_{k-1/2}$, which is still random sample $\omega$ dependent. 
Nevertheless, we can find a constant upper bound for the dual norm of gradient since $\tilde{X}(\omega) \in \mathcal{X}_a$ for all $\omega \in \Omega$, as given in $(b)$. 
We denote $\bar{\nabla}^2_i \coloneqq \max_{x \in \mathcal{X}_a} \norm{\nabla_{x}J^i(x)}^2_*$ in $(c)$. 
For brevity, let $C_g \coloneqq (\sum_{i \in \mathcal{N}} n^i\bar{\nabla}_i\bar{u}^i_*)^2$ and $\norm{G_k}^2_* \leq C_g/(\delta_k)^2 \cdot \norm{\hat{X}_{k+1/2} - \hat{X}_{k-1/2}}^2$. 

We next investigate the property of $\norm{\hat{X}_{k+1/2} - \hat{X}_{k-1/2}}^2$:
\begin{align*}
& \norm{\hat{X}_{k+1/2} - \hat{X}_{k-1/2}}^2 = \norm{X_{k+1/2} - X_{k-1/2} + \\
& \quad \delta_k \underbrace{R^{-1} (p - X_{k+1/2} + Ru_k)}_{\varphi_k} - \delta_{k-1} \underbrace{R^{-1}(p - X_{k-1/2} + Ru_{k-1})}_{\varphi_{k-1}}}^2 \\
& \leq (1+\alpha_1)\norm{X_{k+1/2} - X_{k-1/2}}^2 + (1+\frac{1}{\alpha_1})\delta_k^2\norm{\varphi_k - \frac{\delta_{k-1}}{\delta_k}\varphi_{k-1}}^2 \\
& \leq (1+\alpha_1)\norm{X_{k+1/2} - X_{k-1/2}}^2 + (1+\frac{1}{\alpha_1})\delta_k^2 \bar{\varphi}^2, 
\end{align*}
where $p \coloneqq [p_i]_{i \in \mathcal{N}}$, 
$R \coloneqq \blkd{\{r^iI_{n^i}\}_{i \in \mathcal{N}}}$, 
some arbitrary constant $\alpha_1 > 0$, 
and $\bar{\varphi}^2$ denotes a constant upper bound for $\norm{\varphi_k - \frac{\delta_{k-1}}{\delta_k}\varphi_{k-1}}^2$. 
For OMD, stacking across all players, the iterations suggested by Algorithm~\ref{alg:optm-mrdesc} consist of the following two main updating steps:
\begin{align}\label{eq:compact-omd}
X_{k+1/2} = P_{X_k, \mathcal{X}}(-\gamma_kG_{k-1}), \; X_{k+1} = P_{X_k, \mathcal{X}}(-\gamma_kG_{k}).  
\end{align}
So, we get the following inequalities: 
\begin{align*}
& \norm{X_{k+1/2} - X_{k-1/2}}^2 = \norm{X_{k+1/2} - X_{k} + X_{k} - X_{k-1/2}}^2 \\
& \leq 2\norm{X_{k+1/2} - X_{k}}^2 + 2\norm{X_{k} - X_{k-1/2}}^2 \\
& \overset{(a)}{\leq} 2\norm{\nabla\psi^*(\nabla \psi(X_k) - \gamma_k G_{k-1}) - \nabla \psi^*(\nabla \psi(X_k))}^2 + \\
& 2\norm{\nabla\psi^*(\nabla \psi(X_{k-1}) - \gamma_{k-1} G_{k-1}) - \nabla\psi^*(\nabla \psi(X_{k-1}) - \gamma_{k-1} G_{k-2})}^2\\
& \overset{(b)}{\leq} \frac{2}{\tilde{\mu}^2}\gamma_k^2\norm{G_{k-1}}^2_* + \frac{2}{\tilde{\mu}^2}\gamma_{k-1}^2 \norm{G_{k-1} - G_{k-2}}^2_* \\
& \leq \frac{6}{\tilde{\mu}^2}\gamma_{k-1}^2\norm{G_{k-1}}^2_* + \frac{4}{\tilde{\mu}^2}\gamma_{k-1}^2\norm{G_{k-2}}^2_*. 
\end{align*}
where in the above equation, we disregard specifying the feasible set since in OMD, two prox-mappings are all regarding the strategy space $\mathcal{X}$; 
$(a)$ and $(b)$ is are direct results of applying Lemma~\ref{le:distgen-func} $(\romannum{2})$, $(\romannum{5})$, and the Lipschitz continuity in $(\romannum{4})$. 
Combining with the above yields, for all $k \geq 2$, 
\begin{align*}
& \norm{G_k}^2_* \leq \frac{C_g}{\delta_k^2}\Big((1+\alpha_1)\norm{X_{k+1/2} - X_{k-1/2}}^2 + (1+\frac{1}{\alpha_1})\delta_k^2\bar{\varphi}^2\Big) \\
& \leq \frac{2C_g(1+\alpha_1)}{\tilde{\mu}^2}\Big(\frac{\gamma_{k-1}}{\delta_k}\Big)^2 \big(3\norm{G_{k-1}}^2_* + 2\norm{G_{k-2}}^2_*\big) + C_g(1+\frac{1}{\alpha_1})\bar{\varphi}^2. 
\end{align*}
Since $\lim_{k \to \infty} \gamma_{k-1}/\delta_k = 0$ by Lemma~\ref{le:seq-convg}, for an arbitrary $\varepsilon > 0$, there exists a constant index $K$ such that for all $k > K$, $\frac{2C_g(1+\alpha_1)}{\tilde{\mu}^2}\Big(\frac{\gamma_{k-1}}{\delta_k}\Big)^2 < \varepsilon$ and $\norm{G_k}^2_* \leq 3\varepsilon \norm{G_{k-1}}^2_* + 2\varepsilon \norm{G_{k-2}}^2_* + C_{\bar{g}}$, with $C_{\bar{g}} \coloneqq C_g(1+\frac{1}{\alpha_1})\bar{\varphi}^2$. 
For $k > K$, by the Jury's test and the characteristic polynomial of the linear discrete-time system above $Q(\lambda) = \lambda^2 - 3\varepsilon\lambda - 2\varepsilon$, the system is stable if $Q(1) > 0$, $Q(-1) > 0$, and $\mid -2\varepsilon \mid < 1$, which together imply that it suffices to have $\varepsilon < 1/5$. 
For $k \leq K$, $\norm{G_k}_* \leq \norm{F(X_{k+1/2})}_* + \alpha_B\delta_k + \alpha_V/(\delta_K)^2$. 
Consequently, we have $\sup_{k \in \nset{}{+}}\norm{G_k}^2_* < \infty$, and there exists a constant $C_V$ such that $\sup_{k \in \nset{}{+}}\norm{V_k}^2_* \leq C_V$.

For RMD, the compact formulation of Algorithm~\ref{alg:refl-desc} can be written as
\begin{align}\label{eq:compact-rmd}
\begin{split}
& X_{k+1/2} = P_{X_{k}, \rset{n}{}}(-(\nabla \psi(X_{k-1}) - \nabla \psi(X_k))), \\
& X_{k+1} = P_{X_{k}, \mathcal{X}}(-\gamma_k G_k). 
\end{split}
\end{align}
Since the two prox-mappings in RMD are implemented regarding two different sets, we consider a looser upper bound for $\norm{X_{k+1/2} - X_{k-1/2}}^2$ compared with that of OMD: 
\begin{align*}
& \norm{X_{k+1/2} - X_{k-1/2}}^2  = \norm{X_{k+1/2} - X_{k} + X_{k} - X_{k-1} + X_{k-1} -  X_{k-1/2}}^2\\
& = 4\norm{X_{k+1/2} - X_{k}}^2 + 4\norm{X_{k} - X_{k-1}}^2 + 2\norm{X_{k-1} -  X_{k-1/2}}^2\\
& \overset{(a)}{\leq} (\tilde{L}/\tilde{\mu})^2 \cdot (8\norm{X_{k} - X_{k-1}}^2 + 2\norm{X_{k-1} -  X_{k-2}}^2), 
\end{align*}
where $(a)$ follows from the fact that the global distance generating function $\psi$ enjoys $\tilde{\mu}$-strong monotonicity and $\tilde{L}$-smoothness.
We then use the second update in the reflected gradient to get:
\begin{align*}
\norm{X_{k} - X_{k-1}}^2 \leq \frac{\gamma_{k-1}^2}{\tilde{\mu}^2}\norm{G_{k-1}}^2_*, 
\norm{X_{k-1} - X_{k-2}}^2 \leq \frac{\gamma_{k-2}^2}{\tilde{\mu}^2}\norm{G_{k-2}}^2_*.
\end{align*}
By substituting $\norm{\hat{X}_{k+1/2} - \hat{X}_{k-1/2}}^2$ in the upper bound for $\norm{G_k}^2_*$, we can obtain the following relation: 
\begin{align*}
\norm{G_k}^2_* \leq \frac{2C_g\tilde{L}^2(1+\alpha_1)}{\tilde{\mu}^4}\Big(\frac{\gamma_{k-2}}{\delta_k}\Big)^2\big(4\norm{G_{k-1}}^2_* + \norm{G_{k-2}}^2_*\big) + C_{\bar{g}}.  
\end{align*}
Similarly, by Lemma~\ref{le:seq-convg}, $\lim_{k \to \infty}\gamma_{k-2} / \delta_k = 0$, for an arbitrary $\varepsilon > 0$.
Hence, there exists a constant index $K$ such that for all $k > K$, $\frac{2C_gL^2(1+\alpha_1)}{\tilde{\mu}^4}\Big(\frac{\gamma_{k-2}}{\delta_k}\Big)^2 < \varepsilon$, and $\norm{G_k}^2_* \leq 4\varepsilon \norm{G_{k-1}}^2_* + \varepsilon \norm{G_{k-2}}^2_* + C_{\bar{g}}$. 
Likewise, by Jury's stability criterion and the characteristic polynomial $Q(\lambda) = \lambda^2 - 4\varepsilon\lambda - \varepsilon$, the condition $\varepsilon < 1/5$ can ensure that $Q(1) > 0$, $Q(-1) > 0$, and $\mid -\varepsilon \mid < 1$. 
The remaining arguments are the same as those in OMD, and we arrive at the desired conclusion that for RMD, there exists a constant $C_V$ such that $\sup_{k \in \nset{}{+}}\norm{V_k}^2_* \leq C_V$. 
\end{proof}

\subsection{Almost-Sure Convergence of the Proposed Algorithms in Pseudo-Monotone Plus Games}\label{appd:as-convg}

\begin{proof}
(Proof of Theorem~\ref{thm:asconvg}) \\
\textbf{Algorithm~\ref{alg:optm-mrdesc}.} 
Combining the compact formulations for OMD in \eqref{eq:compact-omd} and \eqref{eq:stdeq-2} in Lemma~\ref{le:stdeq}, we procure the following recurrent relation:
\begin{align*}
D(p, X_{k+1}) &\leq D(p, X_{k}) - \gamma_k\langle G_{k}, X_{k+1/2} - p\rangle \\
& \qquad + \frac{\gamma_k^2}{2\tilde{\mu}}\norm{G_{k} - G_{k-1}}^2_* - \frac{\tilde{\mu}}{2}\norm{X_{k+1/2} - X_{k}}^2,
\end{align*}
for arbitrary $p \in \mathcal{X}$. 
Recall that under the current choices of parameters \eqref{eq:merelm-params}, it has been proved in Lemma~\ref{le:bdd-stoch-err} that $\norm{G_k}_*$ is bounded for all $k \in \nset{}{}$. 
Hence, $\frac{\gamma_k^2}{2\Tilde{\mu}}\norm{G_k - G_{k-1}}^2_* \leq \frac{\gamma_k^2}{\Tilde{\mu}}\norm{G_k}^2_* + \frac{\gamma_k^2}{\Tilde{\mu}}\norm{G_{k-1}}^2_*$ decays at the same rate as $\gamma_k^2$, and we can infer that it is summable. 
Also note that, given the $\sigma$-field $\mathcal{F}_k \coloneqq \sigma\{X_0, u_1, \ldots, u_{k-1}\}$, $X_{k+1}$, $G_k$, and $V_k$ are the only three variables from above that are not $\mathcal{F}_k$-measurable. 
Thus, taking the conditional expectation $\expt{}{\cdot \mid \mathcal{F}_k}$ yields:
\begin{align*}
& \expt{}{D(p, X_{k+1}) \mid \mathcal{F}_k} \overset{(a)}{\leq} D(p, X_{k}) - \gamma_k\langle F(X_{k+1/2}), X_{k+1/2} - p \rangle \\
& \qquad + \gamma_k \norm{B_k}_* \cdot \norm{X_{k+1/2} - p} - \frac{\Tilde{\mu}}{2} \norm{X_{k+1/2} - X_{k}}^2  \\
& \qquad + \frac{\gamma_k^2}{\Tilde{\mu}}\expt{}{\norm{G_k}^2_* \mid \mathcal{F}_k} + \frac{\gamma_k^2}{\Tilde{\mu}}\norm{G_{k-1}}^2_* \\
\begin{split}
& \overset{(b)}{\leq} D(p, X_{k}) - \gamma_k\langle F(X_{k+1/2}), X_{k+1/2} - p \rangle - \frac{\Tilde{\mu}}{2} \norm{X_{k+1/2} - X_{k}}^2 \\
& \qquad + \alpha_B D_{\mathcal{X}}\gamma_k\delta_k + \frac{\gamma_k^2}{\Tilde{\mu}}\expt{}{\norm{G_k}^2_* \mid \mathcal{F}_k} + \frac{\gamma_k^2}{\Tilde{\mu}}\norm{G_{k-1}}^2_*, 
\end{split}
\stepcounter{equation}\tag{\theequation}\label{eq:optm-bd}
\end{align*}
where $(a)$ is a result of the decomposition $G_k = F(X_{k+1/2}) + B_k + V_k$, 
the equality that $\expt{}{\langle V_k, X_{k+1/2} - p\rangle \mid \mathcal{F}_k} = \langle \expt{}{V_k \mid \mathcal{F}_k}, X_{k+1/2} - p\rangle = 0$, and applying the Cauchy-Schwarz inequality to $\langle B_k, X_{k+1/2} - p\rangle$; 
for $(b)$, we use the fact that $X_k, X_{k-1/2} \in \mathcal{X}$ and the ranges of random variables $\norm{B_k}_*$ and $\norm{V_k}^2_*$ satisfy $\norm{B_k}_* \leq \alpha_B \delta_k$ and $\norm{V_k}^2_* \leq C_V$, respectively. 

We then replace $p$ with an arbitrary critical point $x_*$ of the game $\mathcal{G}$ under study, and hence $\langle F(X_{k+1/2}), X_{k+1/2} - x_*\rangle \geq 0$ for all possible values of $X_{k+1/2}$ by pseudomonotonicity. 
Using the Robbins-Siegmund (R-S) Theorem \cite[Thm.~1]{robbins1971convergence}, we can conclude with the following claims: 
\begin{outline}[enumerate]
\1 $(D(x_*, X_k))_{k \in \nset{}{}}$ converges to an a.s. finite limit;
\1 $\sum_{k \in \nset{}{+}} \Tilde{\mu}/2\norm{X_{k+1/2} - X_{k}}^2 < \infty$ a.s.;
\1 $\sum_{k \in \nset{}{}} \gamma_k\langle F(X_{k+1/2}), X_{k+1/2} - x_* \rangle < \infty$ a.s. 
\end{outline}
Claim $(\romannum{2})$ suggests that $\lim_{k \to \infty}\norm{X_{k+1/2} - X_{k}}^2 = 0$ a.s. 
Moreover, since $(\gamma_k)_{k \in \nset{}{+}}$ is not summable we obtain $\liminf_{k \to \infty} \langle F(X_{k+1/2}), X_{k+1/2} - x_*\rangle (\omega) = 0$
from claim $(\romannum{3})$, for any $\omega \in \tilde{\Omega}$ where $\tilde{\Omega} \subseteq \Omega$ has probability one. 
In other words, along a subsequence $(k_m)_{m \in \nset{}{}} \subseteq \nset{}{}$, we have $\langle F(X_{k_m + 1/2}), X_{k_m+1/2} - x_*\rangle(\omega) \to 0$. 
Since $(X_{k}(\omega))_{k \in \nset{}{}} \in \mathcal{X}$ is a bounded sequence, there exists a subsubsequence with $(\ell_m)_{m \in \nset{}{}} \subseteq (k_m)_{m \in \nset{}{}}$ such that $(X_{\ell_m}(\omega))_{m \in \nset{}{}}$ converges to a point $X^\dagger(\omega) \in \mathcal{X}$. 
Claim $(\romannum{2})$ further implies that $X_{\ell_m + 1/2}(\omega) \to X^\dagger(\omega)$. 
By the continuity of $F$, we have:
\begin{align*}
& \lim_{m\to\infty} \langle F(X_{\ell_m+1/2}(\omega)), X_{\ell_m+1/2}(\omega) - x_* \rangle \\
& = \langle \lim_{m\to\infty}F(X_{\ell_m+1/2}(\omega)), \lim_{m\to\infty}(X_{\ell_m+1/2}(\omega) - x_*) \rangle \\
& = \langle F(X^\dagger(\omega)), X^\dagger(\omega) - x_* \rangle = 0
\end{align*}
Since $x_*$ is a solution of the associated VI, we have $\langle F(x_*), X^\dagger(\omega) - x_* \rangle \geq 0$. 
Combining the above results, by the assumed pseudo-monotone plus property of $F$, this implies that $F(x_*) = F(X^\dagger(\omega))$, which further suggests $\forall x \in \mathcal{X}$:
\begin{align*}
& \langle F(X^\dagger(\omega)), x - X^\dagger(\omega)\rangle = \langle F(X^\dagger(\omega)), x - x_* + x_* - X^\dagger(\omega)\rangle \\
& = \langle F(x_*), x - x_*\rangle + \langle F(X^\dagger(\omega)), x_* - X^\dagger(\omega) \rangle \geq 0, 
\end{align*}
i.e., $X^\dagger(\omega)$ is a solution to the associated VI and hence a critical point of the original game under study. 
We can then replace $x_*$ in the recurrent inequality with $X^\dagger(\omega)$. 
Combining it with the fact that there exists a subsequence $(X_{\ell_m})_{m \in \nset{}{}}$ such that $D(X^\dagger(\omega), X_{\ell_m}(\omega)) \to 0$ by the Bregman reciprocity assumption and the fact that $(D(X^\dagger(\omega), X_{\ell_m}(\omega)))_{k \in \nset{}{}}$ admits a finite limit, we can conclude that $D(X^\dagger(\omega), X_{k}(\omega)) \to 0$ for the whole sequence and hence $X_{k}(\omega) \to X^\dagger(\omega)$. 
Finally, recall that $\hat{X}_{k+1/2} = (1 - \delta_kR^{-1})X_{k+1/2} + \delta_kR^{-1}(p + Ru_{k})$ with $p \coloneqq [p_i]_{i \in \mathcal{N}}$ and $R \coloneqq \blkd{\{r^iI_{n^i}\}_{i \in \mathcal{N}}}$, and thus we have $\hat{X}_{k+1/2} \overset{\text{a.s.}}{\to} X_{k+1/2}$ as $\delta_k \to 0$, which completes our proof of a.s. convergence of the sequence of realized actions $(\hat{X}_{k+1/2})_{k \in \nset{}{+}}$. 
\\\\
\textbf{Algorithm~\ref{alg:refl-desc}.} Recall the compact formulation for RMD in \eqref{eq:compact-rmd}. 
For an arbitrary $p \in \mathcal{X}$, substituting the variables in \eqref{eq:stdeq-2} with the primal and dual variables above yields:
\begin{align}\label{eq:rfltdesc-stdineq}
\begin{split}
D(p, X_{k+1}) & \leq D(p, X_{k}) - \gamma_k \langle G_k, X_{k+1/2} - p\rangle - \frac{\tilde{\mu}}{2}\norm{X_{k+1/2} - X_{k}}^2 \\
& + \frac{1}{2\tilde{\mu}} \norm{-\gamma_k G_k + (\nabla \psi(X_{k-1}) - \nabla \psi(X_k))}^2_*. 
\end{split}
\end{align}
To apply the R-S theorem and prove the a.s. convergence, we dissect the last dual norm from the above inequality and show that its summability can be well-controlled by the step size $\gamma_k$: 
\begin{align*}
& \frac{1}{2\tilde{\mu}} \norm{-\gamma_k G_k + (\nabla \psi(X_{k-1}) - \nabla \psi(X_k))}^2_* \\
& \leq \frac{(\gamma_k)^2}{\tilde{\mu}}\norm{G_k}^2_* + \frac{1}{\tilde{\mu}}\norm{\nabla \psi(X_{k-1}) - \nabla \psi(X_k)}^2_* \\
& \overset{(a)}{\leq} \frac{(\gamma_k)^2}{\tilde{\mu}}\norm{G_k}^2_* + \frac{\Tilde{L}^2}{\tilde{\mu}}\norm{X_{k-1} - X_k}^2 \overset{(b)}{=} \frac{(\gamma_k)^2}{\tilde{\mu}}\norm{G_k}^2_* \\
&  \qquad + \frac{\Tilde{L}^2}{\tilde{\mu}}\norm{\nabla \psi^*(\nabla \psi(X_{k-1}) - \gamma_{k-1}G_{k-1}) - \nabla \psi^*(\nabla \psi(X_{k-1}))}^2 \\
& \overset{(c)}{\leq} \frac{(\gamma_k)^2}{\tilde{\mu}}\norm{G_k}^2_* + \frac{(\Tilde{L}\gamma_{k-1})^2}{\tilde{\mu}^3}\norm{G_{k-1}}^2_*,
\end{align*}
where $(a)$ directly follows from the $\tilde{L}$-Lipschitz continuity of $\nabla \psi$; 
in $(b)$, we expand the expression of $X_k$ in terms of the mirror map $\nabla \psi^*$ by applying Lemma~\ref{le:distgen-func}$(\romannum{3})$ and note that for arbitrary $x \in \mathcal{X}$, $(\nabla \psi^* \circ \nabla \psi) (x) = x$; $(c)$ is the result of $(1/\tilde{\mu})$-Lipschitz continuity of $\nabla \psi^*$. 
Recall the decomposition $G_k = F(X_{k+1/2}) + V_k + B_k$ and the boundedness of $\norm{G_k}_*$ for all $k \in \nset{}{+}$ as proved in Lemma~\ref{le:bdd-stoch-err}. 
The summability of the last dual norm in \eqref{eq:rfltdesc-stdineq} inherits from that of $(\gamma_k^2)_{k \in \nset{}{+}}$. 
With the filtration defined as $\mathcal{F}_k \coloneqq \sigma\{X_0, u_1, \ldots, u_{k-1}\}$, we take the conditional expectation $\expt{}{\cdot \mid \mathcal{F}_k}$ on both sides of \eqref{eq:rfltdesc-stdineq} and obtain:
\begin{align}\label{eq:rfltdesc-cond-stdineq}
\begin{split}
& \expt{}{D(p, X_{k+1}) \mid \mathcal{F}_k} \leq D(p, X_{k}) - \gamma_k \langle F(X_{k+1/2}), X_{k+1/2} - p\rangle \\
& \qquad - \tilde{\mu}/2\cdot\norm{X_{k+1/2} - X_k}^2 + \gamma_k\norm{B_k}_*\norm{X_{k+1/2} - p} \\
& \qquad + (\gamma_k)^2/\tilde{\mu} \cdot \expt{}{\norm{G_k}^2_* \mid \mathcal{F}_k} + (\Tilde{L}\gamma_{k-1})^2/\tilde{\mu}^3 \cdot \norm{G_{k-1}}^2_*.
\end{split}
\end{align}
Yet, for the inner product between $F(X_{k+1/2})$ and $X_{k+1/2} - p$, $X_{k+1/2} \in \mathcal{X}_R$ and can sit outside $\mathcal{X}$. To leverage the regularity in Definition~\ref{def:cps} and ensure that the inner product is positive, it is easier for us to work with $\langle F(X_k), X_k - p\rangle$ instead, where $X_k \in \mathcal{X}$. 
To this end, we derive an upper bound for the inner product in \eqref{eq:rfltdesc-cond-stdineq} as follows:
$
 - \gamma_k \langle F(X_{k+1/2}), X_{k+1/2} - p\rangle = -\gamma_k\langle F(X_k), X_k - p\rangle 
 + \gamma_k \langle F(X_{k+1/2}) - F(X_{k}), p\rangle + \gamma_k\langle F(X_k), X_k\rangle - \gamma_k\langle F(X_{k+1/2}), X_{k+1/2}\rangle 
 \leq  -\gamma_k\langle F(X_k), X_k - p\rangle + \gamma_kL\norm{X_{k+1/2} - X_{k}}\cdot\norm{p} + \gamma_k\langle F(X_k) - F(X_{k+1/2}), X_k\rangle - \gamma_k\langle F(X_{k+1/2}), X_{k+1/2} - X_k\rangle \leq -\gamma_k\langle F(X_k), X_k - x_*\rangle + \Delta^\prime_k
$, where here $\Delta^\prime_k \coloneqq \gamma_k\norm{X_{k+1/2} - X_k}\cdot(L\norm{p} + L\norm{X_k} + \norm{F(X_{k+1/2})}_*)$.
By combining the relation $\norm{X_{k+1/2} - X_{k}} \leq \frac{\Tilde{L}}{\Tilde{\mu}}\norm{X_{k} - X_{k-1}} \leq \frac{\Tilde{L}\gamma_{k-1}}{\Tilde{\mu}}\norm{G_{k-1}}_*$ and the boundedness of $(L\norm{x_*} + L\norm{X_k} + \norm{F(X_{k+1/2})}_*)$, we claim that $\Delta^\prime_k$ is summable and obtain the following recurrent inequality
\begin{align}\label{eq:rfltdesc-cond-stdineq2}
\begin{split}
& \expt{}{D(p, X_{k+1}) \mid \mathcal{F}_k} \leq D(p, X_{k}) - \gamma_k \langle F(X_{k}), X_{k} - p\rangle \\
& \qquad - \tilde{\mu}/2\cdot\norm{X_{k+1/2} - X_k}^2 + \Delta^\prime_k + \alpha_B D_{\mathcal{X}}\gamma_k\delta_k \\
& \qquad + (\gamma_k)^2/\tilde{\mu} \cdot \expt{}{\norm{G_k}^2_* \mid \mathcal{F}_k} + (\Tilde{L}\gamma_{k-1})^2/\tilde{\mu}^3 \cdot \norm{G_{k-1}}^2_*.
\end{split}
\end{align}
Once again, we substitute a critical point $x_* \in \mathcal{X}$ for $p$ and arrive at the following claims by applying the R-S theorem:
\begin{outline}[enumerate]
\1 $(D(x_*, X_k))_{k \in \nset{}{}}$ converges to an a.s. finite limit;
\1 $\sum_{k \in \nset{}{+}} \tilde{\mu}/2\norm{X_{k+1/2} - X_{k}}^2 < \infty$ a.s.;
\1 $\sum_{k \in \nset{}{}} \gamma_k\langle F(X_{k}), X_{k} - x_* \rangle < \infty$ a.s. 
\end{outline}
Under the assumption that the operator $F$ is pseudo-monotone plus, the remaining arguments of OMD can be directly carried over into the discussion here; the same result about the a.s. convergence of $(\hat{X}_{k+1/2})_{k \in \nset{}{+}}$ can then be justified. 
\end{proof}

\begin{proof}
(Proof of Corollary~\ref{coro:variant-asconvg})
For OMD and RMD, we apply the same arguments in Theorem~\ref{thm:asconvg} to obtain the three claims and the existence of the (sub)subsequence $(\ell_m)_{m \in \nset{}{}}$ after using the R-S theorem. 
Along $(\ell_m)_{m \in \nset{}{}}$, we have $X_{\ell_m + 1/2}(\omega) \to X^\dagger(\omega)$, $X_{\ell_m}(\omega) \to X^\dagger(\omega)$, and the limit point $X^\dagger(\omega)$ satisfies $\langle F(X^\dagger(\omega)), X^\dagger(\omega) - x_*\rangle = 0$. 
If $X^\dagger(\omega)$ is a critical point, then the remaining statements of Theorem~\ref{thm:asconvg} can again be applied here to conclude the a.s. convergence of $(\hat{X}_{k+1/2})_{k \in \nset{}{+}}$. 
In the remaining proof, we will show that $X^\dagger(\omega)$ is a critical point of $\mathcal{G}$ under Def.~\ref{def:regu} $(\romannum{3}) - (\romannum{6})$.

For Def.~\ref{def:regu} $(\romannum{5})$, since $\langle F(x), x - x_*\rangle > 0$ for all $x \in \mathcal{X}\backslash \mathcal{X}_*$ while $\langle F(X^\dagger(\omega)), X^\dagger(\omega) - x_*\rangle = 0$, we have that $X^\dagger(\omega) \in \mathcal{X}_*$. 

For Def.~\ref{def:regu} $(\romannum{3})$ and $(\romannum{4})$, $X^\dagger(\omega) \in \mathcal{X}_*$ directly follows from the fact that either strict or strong pseudo-monotonicity is a sufficient condition for strict coherence. 

For Def.~\ref{def:regu} $(\romannum{6})$, given that $x_*$ is a critical point, we have $\Phi(x_*) \leq \Phi(x)$ for all $x\in \mathcal{X}$, i.e., it is a global minimum. 
The relation $\langle F(X^\dagger(\omega)), x_* - X^\dagger(\omega)\rangle  = 0$ further implies $\Phi(x_*) \geq \Phi(X^\dagger(\omega))$. 
We can thus conclude that $X^\dagger(\omega)$ is a global minimum for $\Phi$ and hence a critical point. 
\end{proof}

\subsection{Ergodic Convergence Rates in Merely Monotone Games}\label{appd:ergd-convg}

\begin{proof}
(Proof of Theorem~\ref{thm:erg-rate}) 
\\
\textbf{Algorithm~\ref{alg:optm-mrdesc}.} 
Under Assumptions~\ref{asp:objt-set} and \ref{asp:lipschitz}, we can obtain the recurrent inequality \eqref{eq:optm-bd}. 
From the proof of Lemma~\ref{le:bdd-stoch-err}, there exists a constant $\bar{g}$ such that $\sup_{k \in \nset{}{+}, \omega \in \Omega} \norm{G_k(\omega)}_* \leq \bar{g}$ and $\bar{\nabla} \coloneqq \max_{x \in \mathcal{X}}\norm{F(x)}_*$. 
Taking expectations of both sides of \eqref{eq:optm-bd} gives: 
\begin{align*}
\expt{}{\gamma_k \langle F(X_{k+1/2}), X_{k+1/2} - p\rangle} \leq \expt{}{D(p, X_k)} - \expt{}{D(p, X_{k+1})} + \Delta_k, 
\end{align*}
where we let $\Delta_k \coloneqq \alpha_B D_{\mathcal{X}}\gamma_k\delta_k + 2\gamma_k^2\bar{g}^2/\tilde{\mu}$ for notational simplicity. 
Since $F$ is a monotone operator over $\mathcal{X}$, we have $\langle F(p), \hat{X}_{k+1/2} - p \rangle \leq \langle F(p), X_{k+1/2} - p \rangle + \norm{F(p)}_*\norm{\hat{X}_{k+1/2} - X_{k+1/2}} \leq \langle F(X_{k+1/2}), X_{k+1/2} - p \rangle + \bar{\nabla}\alpha_{\mathcal{X}}\delta_k$, with $\alpha_{\mathcal{X}}$ denoting some constant depending on the geometry of $\mathcal{X}$. 
Let $\Delta^\star_k \coloneqq \Delta_k + \bar{\nabla}\alpha_{\mathcal{X}}\delta_k\gamma_k$
Now, by a simple telescoping sum from $t=1$ to $k$, we get:
\begin{align*}
\frac{1}{\sum_{t=1}^k \gamma_t}\sum_{t=1}^k \expt{}{\langle F(p), \gamma_t(\hat{X}_{t+1/2} - p)\rangle} \leq \frac{\expt{}{D(p, X_1)} + \sum_{t=1}^{k} \Delta^\star_t}{\sum_{t=1}^k \gamma_t}. 
\end{align*}
Under the assumption of $\gamma_k$ and $\delta_k$ in \eqref{eq:merelm-params}, there exists a constant $M$ such that the incremental sequence $\sum_{t=1}^k\Delta^\star_t \leq M$ for all $k \in \nset{}{+}$. 
Lastly, applying the definitions in \eqref{eq:merit-func} and \eqref{eq:erg-avg} and letting $p^* \coloneqq \argmax_{p \in \mathcal{X}} \langle F(p), \Check{X}_{k} - p\rangle$ readily yields:
\begin{align*}
\expt{}{\text{Err}_{\mathcal{X}}(\check{X}_k)} \leq \frac{\expt{}{D(p^*, X_1)} + M}{\sum_{t=1}^k \gamma_t} \leq \frac{\expt{}{\max_{p \in \mathcal{X}} D(p, X_1)} + M}{\sum_{t=1}^k \gamma_t},
\end{align*}
and our proof for OMD is complete. 
\\\\
\textbf{Algorithm~\ref{alg:refl-desc}.} 
In the same vein as the proof of Theorem~\ref{thm:asconvg}, we readily obtain the following building block from \eqref{eq:rfltdesc-cond-stdineq2} for telescoping:
\begin{align*}
\expt{}{\gamma_k\langle F(X_{k}), X_{k} - p\rangle} \leq \expt{}{D(p, X_k)} - \expt{}{D(p, X_{k+1})} + \Delta_{k}, 
\end{align*}
where $\Delta_k \coloneqq \alpha_BD_{\mathcal{X}_a}\gamma_k\delta_k + (\gamma_k\bar{g})^2/\Tilde{\mu} + (\Tilde{L}\gamma_{k-1}\Bar{g})^2/\Tilde{\mu}^3 + \Delta^\prime_k$. 
By the monotonicity of $F$ over $\mathcal{X}$, we have $\langle F(p), \hat{X}_{k+1/2} - p\rangle \leq \norm{F(p)}_* (\norm{\hat{X}_{k+1/2} - {X}_{k+1/2}} + \norm{{X}_{k+1/2} - {X}_{k}}) + \langle F(p), X_k - p\rangle \leq \bar{\nabla}\alpha_{\mathcal{X}}\delta_k + \bar{\nabla}\Tilde{L}\Bar{g}\gamma_{k-1}/\Tilde{\mu} + \langle F(X_k), X_k - p\rangle$, with $\alpha_{\mathcal{X}}$ defined as above. 
Let $\Delta^\star_k \coloneqq \Delta_k + \bar{\nabla}\alpha_{\mathcal{X}}\delta_k\gamma_k + \bar{\nabla}\Tilde{L}\Bar{g}\gamma_k\gamma_{k-1}/\Tilde{\mu}$
Likewise, under \eqref{eq:merelm-params}, there exists a constant $M$ such that the incremental sequence $\sum_{t=1}^k\Delta^\star_t \leq M$ for all $k \in \nset{}{+}$. 
The remaining arguments resemble the above ones for OMD, except that now the constant $M$ admits a different value. 
\end{proof}

\subsection{$O(1/k^{1-\epsilon})$ Convergence Rate of the Proposed Algorithms in Strongly Pseudo-Monotone Games}\label{appd:strg-mono-rate}

\begin{appdxlemma}\label{le:recr-lin-ineq}
Let $(a_k)_{k \in \nset{}{+}}$ be a non-negative sequence. 
The power constants $s, t$ satisfy $ 0 < s < t < 1$ and $s + t > 1$. 
If the sequence $(a_k)_{k \in \nset{}{+}}$ satisfy the recurrent linear inequality
\begin{align*}
    a_{k+1} \leq (1 - \frac{c}{k^s})a_k + \frac{d}{k^{t+s}}, 
\end{align*}
for some positive constant $c$ and $d$, then
\begin{align*}
    a_{k} \leq \frac{c_\star}{k^{t+s-1}} + \frac{d_\star}{k}, \forall k \geq K, 
\end{align*}
where the constant index $K$ satisfies $K > cK^{1-s} \geq 1$; the two coefficients are selected as: $c_\star \coloneqq \frac{d}{\floor{cK^{1-s}} - (t+s-1)}$ and $d_\star \coloneqq \max\{0, \frac{K(K-1)\Tilde{a}_0}{K - \floor{cK^{1-s}}}\}$ with $\Tilde{a}_0 \coloneqq a_K - \frac{c_\star}{K^{s+t-1}}$. 
\end{appdxlemma}
\begin{proof}
The following is largely inspired by that of \cite[Lem.~1]{chung1954stochastic}. 
We can find a $K$ such that $K > cK^{1-s} \geq 1 > t+s-1$. 
For simplicity, let $p \coloneqq t+s-1$ and $\Tilde{c} \coloneqq \floor{cK^{1-s}}$, where $\Tilde{c} - p > 0$ and $\Tilde{c}/K < 1$. 
Starting from the index $K$, the recursive inequality can be relaxed to be $a_{k+1} \leq (1 - \frac{\Tilde{c}}{k})a_k + \frac{d}{k^{t+s}}$. 
Moreover, note that $\frac{1}{k^p} - \frac{1}{(k+1)^p} \leq \frac{p}{k^{p+1}}$, which implies $\frac{1}{(k+1)^p} - (1 - \frac{\Tilde{c}}{k})\frac{1}{k^p} = \frac{\Tilde{c}}{k^{p+1}} - (\frac{1}{k^p} - \frac{1}{(k+1)^p}) \geq \frac{\Tilde{c} - p}{k^{p+1}}$. 
Substituting $d/k^{p+1}$ with the results above yields:
\begin{align*}
a_{k+1} - \frac{d}{\Tilde{c} - p}\cdot\frac{1}{(k+1)^p} \leq (1 - \frac{\Tilde{c}}{k})(a_k - \frac{d}{\Tilde{c} - p}\cdot\frac{1}{k^p}). 
\end{align*}
If $a_K \leq \frac{d}{\Tilde{c} - p}\cdot\frac{1}{K^p}$, then for all $k \geq K$, $a_k \leq \frac{d}{\Tilde{c} - p}\cdot\frac{1}{k^p}$. 
Otherwise, denote $\Tilde{a}_k = a_{k+K} - \frac{d}{\Tilde{c} - p}\cdot\frac{1}{(k + K)^p}$ for $k \in \nset{}{}$ and we have $\Tilde{a}_{k+1} \leq \frac{k+K-\Tilde{c}}{k+K}\Tilde{a}_k$. 
By telescoping, $\Tilde{a}_{k} \leq \prod_{\ell=1}^{\Tilde{c}\wedge k}\frac{K - \Tilde{c} - 1 + \ell}{K + k - \ell}\Tilde{a}_0 \leq \frac{(K-1)\Tilde{a}_0}{k+(K - \Tilde{c})} \leq \frac{K(K-1)\Tilde{a}_0}{K - \Tilde{c}}\cdot\frac{1}{k+K}$. 
Combining all the results above, we can arrive at the desired bound on $a_k$ as stated in the lemma. 
\end{proof}

\begin{remark}
The results here are more conservative compared to \cite[Lem.~4]{chung1954stochastic} and \cite[Lem.~5]{polyakintroduction}, from which we can deduce that the sequence $(a_k)_{k \in \nset{}{+}}$ considered in Lemma~\ref{le:recr-lin-ineq} has an asymptotic convergence rate of $O(1/k^t)$.
Nevertheless, the results from \cite{chung1954stochastic, polyakintroduction} are in a "$\limsup$" sense and it is not as helpful for explicit convergence characterization. 
In light of this, we leverage the results in Lemma~\ref{le:recr-lin-ineq} where the sequence converges at an explicit rate after a certain fixed iteration $K$. 
\end{remark}

\begin{proof}
(Proof of Theorem~\ref{thm:strmon-eucl-convgrate}) 
\\
\textbf{Algorithm~\ref{alg:optm-mrdesc}.} 
By the strong pseudo-monotonicity of the pseudogradient $F$, we can obtain 
\begin{align*}
& \langle F(X_{k+1/2}), X_{k+1/2} - x_* \rangle \geq \mu\norm{X_{k+1/2} - x_*}^2 \geq \\
& \frac{\mu}{2}\norm{X_{k} - x_*}^2 - \mu\norm{X_{k+1/2} - X_{k}}^2 
\overset{(a)}{\geq} \frac{\mu}{\tilde{L}}D(x_*, X_k) - \mu\norm{X_{k+1/2} - X_{k}}^2,
\end{align*}
where $(a)$ follows from our norm-like restriction in Assumption~\ref{asp:distgenfunc-lipsc}. 
Then, by the Lipschitz continuity of the mirror map, the squared norm $\norm{X_{k+1/2} - X_{k}}^2$ can be bounded as follows:
\begin{align*}
\norm{X_{k+1/2} - X_k}^2 &\leq \norm{\nabla \psi^*(\nabla \psi (X_k) - \gamma_k G_{k-1}) - \nabla \psi^*(\nabla \psi(X_k))}^2 \\
& \leq \frac{1}{\tilde{\mu}^2}\gamma_k^2 \norm{G_{k-1}}^2_*.
\end{align*}
Combining the relation above with \eqref{eq:optm-bd} and taking expectations of both sides of the inequality gives
\begin{align*}
\expt{}{D(x_*, X_{k+1})} \leq (1 - \frac{\mu\gamma_k}{\tilde{L}})\expt{}{D(x_*, X_{k})} + C_{e,1}\gamma_k\delta_k,
\end{align*}
where the last term $C_{e,1}\gamma_k\delta_k$ takes care of the other terms in \eqref{eq:optm-bd} that decay at a faster rate, by assuming $\gamma_k = c_\gamma / (k + b_\gamma)^{a_\gamma}$ and $\delta_k = c_\delta / (k + b_\delta)^{a_\delta}$ with $0 < a_{\delta} < a_{\gamma} < 1$ and $a_{\gamma} + a_{\delta} > 1$.
Following Lemma~\ref{le:recr-lin-ineq}, we have that $\expt{}{D(x_*, X_{k})}$ decays at the rate of $1/k^{a_\gamma + a_\delta - 1}$ starting from certain iteration $K$. 
Recall from the analysis above and the construction of perturbed action that  $\norm{X_{k+1/2} - X_k}^2 \leq \frac{1}{\tilde{\mu}^2}\gamma_k^2 \norm{G_{k-1}}^2_* = C_{e,2}\gamma_k^2$ and $\norm{\hat{X}_{k+1/2} - X_{k+1/2}}^2 = \norm{\delta_k u_k + \delta_kR^{-1}(p - X_{k+1/2})}^2 = C_{e,3}\delta_k^2$ for some constants $C_{e,2}$ and $C_{e,3}$.
Due to the fact that $D(p, x) \geq \tilde{\mu}/2\norm{p - x}^2$ for all $p$ and $x$, we readily have $\expt{}{\norm{\hat{X}_{k+1/2} - x_*}^2} = M_1/k^{a_\gamma + a_\delta - 1} + M_2/k$ for all $k > K$, where $M_1$, $M_2$, and $K$ represent some constants determined by the properties of $\mathcal{G}$ as well as $\gamma_k$ and $\delta_k$ chosen.  
\\\\
\textbf{Algorithm~\ref{alg:refl-desc}.} 
Using the intermediate results from Lemma~\ref{le:bdd-stoch-err}, we first note that $\langle F(X_k), X_k - x_*\rangle \geq \mu\norm{X_k - x_*}^2 \geq \mu/\Tilde{L}D(x_*, X_k)$. 
In the same vein of OMD, we combine \eqref{eq:rfltdesc-cond-stdineq2}, the strong pseudomonotonicity of $F$, and the relation above to obtain
\begin{align*}
\expt{}{D(x_*, X_{k+1})} \leq (1 - \frac{\mu\gamma_k}{\tilde{L}})\expt{}{D(x_*, X_{k})} + C_{e,1}\gamma_k\delta_k,
\end{align*}
with $C_{e,1}\gamma_k\delta_k$ handling all terms decaying no slower than the order of $\gamma_k\delta_k$. 
Carrying over the arguments for OMD, we can thus arrive at the same conclusion for RMD. 
\end{proof}

\bibliographystyle{IEEEtran}
\bibliography{IEEEabrv,references}

\begin{thebibliography}{10}
\providecommand{\url}[1]{#1}
\csname url@samestyle\endcsname
\providecommand{\newblock}{\relax}
\providecommand{\bibinfo}[2]{#2}
\providecommand{\BIBentrySTDinterwordspacing}{\spaceskip=0pt\relax}
\providecommand{\BIBentryALTinterwordstretchfactor}{4}
\providecommand{\BIBentryALTinterwordspacing}{\spaceskip=\fontdimen2\font plus
\BIBentryALTinterwordstretchfactor\fontdimen3\font minus
  \fontdimen4\font\relax}
\providecommand{\BIBforeignlanguage}[2]{{%
\expandafter\ifx\csname l@#1\endcsname\relax
\typeout{** WARNING: IEEEtran.bst: No hyphenation pattern has been}%
\typeout{** loaded for the language `#1'. Using the pattern for}%
\typeout{** the default language instead.}%
\else
\language=\csname l@#1\endcsname
\fi
#2}}
\providecommand{\BIBdecl}{\relax}
\BIBdecl

\bibitem{li2022confluence}
T.~Li, G.~Peng, Q.~Zhu, and T.~Ba{\c{s}}ar, ``The confluence of networks,
  games, and learning a game-theoretic framework for multiagent decision making
  over networks,'' \emph{IEEE Control Systems Magazine}, vol.~42, no.~4, pp.
  35--67, 2022.

\bibitem{maharjan2013dependable}
S.~Maharjan, Q.~Zhu, Y.~Zhang, S.~Gjessing, and T.~Basar, ``Dependable demand
  response management in the smart grid: A {Stackelberg} game approach,''
  \emph{IEEE Transactions on Smart Grid}, vol.~4, no.~1, pp. 120--132, 2013.

\bibitem{zhu2012differential}
Q.~Zhu, Z.~Han, and T.~Ba{\c{s}}ar, ``A differential game approach to
  distributed demand side management in smart grid,'' in \emph{2012 IEEE
  International Conference on Communications (ICC)}.\hskip 1em plus 0.5em minus
  0.4em\relax IEEE, 2012, pp. 3345--3350.

\bibitem{han2012game}
Z.~Han, D.~Niyato, W.~Saad, T.~Ba{\c{s}}ar, and A.~Hj{\o}rungnes, \emph{Game
  theory in wireless and communication networks: theory, models, and
  applications}.\hskip 1em plus 0.5em minus 0.4em\relax Cambridge university
  press, 2012.

\bibitem{zhu2012interference}
Q.~Zhu, Z.~Yuan, J.~B. Song, Z.~Han, and T.~Basar, ``Interference aware routing
  game for cognitive radio multi-hop networks,'' \emph{IEEE Journal on Selected
  Areas in Communications}, vol.~30, no.~10, pp. 2006--2015, 2012.

\bibitem{vu2021fast}
D.~Q. Vu, K.~Antonakopoulos, and P.~Mertikopoulos, ``Fast routing under
  uncertainty: Adaptive learning in congestion games via exponential weights,''
  \emph{Advances in Neural Information Processing Systems}, vol.~34, pp.
  14\,708--14\,720, 2021.

\bibitem{nash1950equilibrium}
J.~F. Nash~Jr, ``Equilibrium points in n-person games,'' \emph{Proceedings of
  the national academy of sciences}, vol.~36, no.~1, pp. 48--49, 1950.

\bibitem{neumann2007theoryofgame}
J.~V. Neumann and O.~Morgenstern, \emph{Theory of Games and Economic Behavior
  (Commemorative Edition).}\hskip 1em plus 0.5em minus 0.4em\relax Princeton,
  NJ, USA: Princeton Univ. Press, 2007.

\bibitem{mertikopoulos2019learning}
P.~Mertikopoulos and Z.~Zhou, ``Learning in games with continuous action sets
  and unknown payoff functions,'' \emph{Mathematical Programming}, vol. 173,
  no.~1, pp. 465--507, 2019.

\bibitem{yi2019operator}
P.~Yi and L.~Pavel, ``An operator splitting approach for distributed
  generalized {Nash} equilibria computation,'' \emph{Automatica}, vol. 102, pp.
  111--121, 2019.

\bibitem{tatarenko2020geometric}
T.~Tatarenko, W.~Shi, and A.~Nedi{\'c}, ``Geometric convergence of gradient
  play algorithms for distributed nash equilibrium seeking,'' \emph{IEEE
  Transactions on Automatic Control}, vol.~66, no.~11, pp. 5342--5353, 2020.

\bibitem{pavel2019distributed}
L.~Pavel, ``Distributed {GNE} seeking under partial-decision information over
  networks via a doubly-augmented operator splitting approach,'' \emph{IEEE
  Transactions on Automatic Control}, vol.~65, no.~4, pp. 1584--1597, 2019.

\bibitem{bianchi2022fast}
M.~Bianchi, G.~Belgioioso, and S.~Grammatico, ``Fast generalized {Nash}
  equilibrium seeking under partial-decision information,'' \emph{Automatica},
  vol. 136, p. 110080, 2022.

\bibitem{huang2022distributed}
Y.~Huang and J.~Hu, ``Distributed computation of stochastic {GNE} with partial
  information: An augmented best-response approach,'' \emph{IEEE Transactions
  on Control of Network Systems}, 2022.

\bibitem{cesa2006prediction}
N.~Cesa-Bianchi and G.~Lugosi, \emph{Prediction, learning, and games}.\hskip
  1em plus 0.5em minus 0.4em\relax Cambridge university press, 2006.

\bibitem{hazan2016introduction}
E.~Hazan \emph{et~al.}, ``Introduction to online convex optimization,''
  \emph{Foundations and Trends{\textregistered} in Optimization}, vol.~2, no.
  3-4, pp. 157--325, 2016.

\bibitem{korpelevich1976extragradient}
G.~M. Korpelevich, ``The extragradient method for finding saddle points and
  other problems,'' \emph{Matecon}, vol.~12, pp. 747--756, 1976.

\bibitem{nemirovski2004prox}
A.~Nemirovski, ``Prox-method with rate of convergence {O}(1/t) for variational
  inequalities with lipschitz continuous monotone operators and smooth
  convex-concave saddle point problems,'' \emph{SIAM Journal on Optimization},
  vol.~15, no.~1, pp. 229--251, 2004.

\bibitem{nesterov2007dual}
Y.~Nesterov, ``Dual extrapolation and its applications to solving variational
  inequalities and related problems,'' \emph{Mathematical Programming}, vol.
  109, no.~2, pp. 319--344, 2007.

\bibitem{juditsky2011solving}
A.~Juditsky, A.~Nemirovski, and C.~Tauvel, ``Solving variational inequalities
  with stochastic mirror-prox algorithm,'' \emph{Stochastic Systems}, vol.~1,
  no.~1, pp. 17--58, 2011.

\bibitem{kannan2019optimal}
A.~Kannan and U.~V. Shanbhag, ``Optimal stochastic extragradient schemes for
  pseudomonotone stochastic variational inequality problems and their
  variants,'' \emph{Computational Optimization and Applications}, vol.~74,
  no.~3, pp. 779--820, 2019.

\bibitem{iusem2017extragradient}
A.~N. Iusem, A.~Jofr{\'e}, R.~I. Oliveira, and P.~Thompson, ``Extragradient
  method with variance reduction for stochastic variational inequalities,''
  \emph{SIAM Journal on Optimization}, vol.~27, no.~2, pp. 686--724, 2017.

\bibitem{zhou2017mirror}
Z.~Zhou, P.~Mertikopoulos, A.~L. Moustakas, N.~Bambos, and P.~Glynn, ``Mirror
  descent learning in continuous games,'' in \emph{2017 IEEE 56th Annual
  Conference on Decision and Control (CDC)}.\hskip 1em plus 0.5em minus
  0.4em\relax IEEE, 2017, pp. 5776--5783.

\bibitem{hsieh2019convergence}
Y.-G. Hsieh, F.~Iutzeler, J.~Malick, and P.~Mertikopoulos, ``On the convergence
  of single-call stochastic extra-gradient methods,'' \emph{Advances in Neural
  Information Processing Systems}, vol.~32, 2019.

\bibitem{gidel2018variational}
G.~Gidel, H.~Berard, G.~Vignoud, P.~Vincent, and S.~Lacoste-Julien, ``A
  variational inequality perspective on generative adversarial networks,''
  \emph{arXiv preprint arXiv:1802.10551}, 2018.

\bibitem{azizian2021last}
W.~Azizian, F.~Iutzeler, J.~Malick, and P.~Mertikopoulos, ``The last-iterate
  convergence rate of optimistic mirror descent in stochastic variational
  inequalities,'' in \emph{Conference on Learning Theory}.\hskip 1em plus 0.5em
  minus 0.4em\relax PMLR, 2021, pp. 326--358.

\bibitem{malitsky2015projected}
Y.~Malitsky, ``Projected reflected gradient methods for monotone variational
  inequalities,'' \emph{SIAM Journal on Optimization}, vol.~25, no.~1, pp.
  502--520, 2015.

\bibitem{cui2016analysis}
S.~Cui and U.~V. Shanbhag, ``On the analysis of reflected gradient and
  splitting methods for monotone stochastic variational inequality problems,''
  in \emph{2016 IEEE 55th Conference on Decision and Control (CDC)}.\hskip 1em
  plus 0.5em minus 0.4em\relax IEEE, 2016, pp. 4510--4515.

\bibitem{flaxman2005online}
A.~D. Flaxman, A.~T. Kalai, and H.~B. McMahan, ``Online convex optimization in
  the bandit setting: gradient descent without a gradient,'' in
  \emph{Proceedings of the Sixteenth Annual ACM-SIAM Symposium on Discrete
  Algorithms}, 2005, p. 385–394.

\bibitem{gasnikov2017stochastic}
A.~V. Gasnikov, E.~A. Krymova, A.~A. Lagunovskaya, I.~N. Usmanova, and F.~A.
  Fedorenko, ``Stochastic online optimization. {Single-point} and multi-point
  non-linear multi-armed bandits. {Convex} and strongly-convex case,''
  \emph{Automation and remote control}, vol.~78, no.~2, pp. 224--234, 2017.

\bibitem{zhang2022new}
Y.~Zhang, Y.~Zhou, K.~Ji, and M.~M. Zavlanos, ``A new one-point
  residual-feedback oracle for black-box learning and control,''
  \emph{Automatica}, vol. 136, p. 110006, 2022.

\bibitem{chen2022improve}
X.~Chen, Y.~Tang, and N.~Li, ``Improve single-point zeroth-order optimization
  using high-pass and low-pass filters,'' in \emph{International Conference on
  Machine Learning}.\hskip 1em plus 0.5em minus 0.4em\relax PMLR, 2022, pp.
  3603--3620.

\bibitem{bravo2018bandit}
M.~Bravo, D.~Leslie, and P.~Mertikopoulos, ``Bandit learning in concave
  {N}-person games,'' \emph{Advances in Neural Information Processing Systems},
  vol.~31, 2018.

\bibitem{tatarenko2020bandit}
T.~Tatarenko and M.~Kamgarpour, ``Bandit online learning of {Nash} equilibria
  in monotone games,'' \emph{arXiv preprint arXiv:2009.04258}, 2020.

\bibitem{tatarenko2022rate}
------, ``On the rate of convergence of payoff-based algorithms to {Nash}
  equilibrium in strongly monotone games,'' \emph{arXiv preprint
  arXiv:2202.11147}, 2022.

\bibitem{lin2021optimal}
T.~Lin, Z.~Zhou, W.~Ba, and J.~Zhang, ``Optimal no-regret learning in strongly
  monotone games with bandit feedback,'' \emph{arXiv preprint
  arXiv:2112.02856}, 2021.

\bibitem{mertikopoulos2022learning}
P.~Mertikopoulos, Y.-P. Hsieh, and V.~Cevher, ``Learning in games from a
  stochastic approximation viewpoint,'' \emph{arXiv preprint arXiv:2206.03922},
  2022.

\bibitem{facchinei2003finite}
F.~Facchinei and J.-S. Pang, \emph{Finite-dimensional variational inequalities
  and complementarity problems}.\hskip 1em plus 0.5em minus 0.4em\relax
  Springer, 2003.

\bibitem{scutari2010convex}
G.~Scutari, D.~P. Palomar, F.~Facchinei, and J.-S. Pang, ``Convex optimization,
  game theory, and variational inequality theory,'' \emph{IEEE Signal
  Processing Magazine}, vol.~27, no.~3, pp. 35--49, 2010.

\bibitem{karamardian1976complementarity}
S.~Karamardian, ``Complementarity problems over cones with monotone and
  pseudomonotone maps,'' \emph{Journal of Optimization Theory and
  Applications}, vol.~18, no.~4, pp. 445--454, 1976.

\bibitem{nemirovskij1983problem}
A.~S. Nemirovskij and D.~B. Yudin, ``Problem complexity and method efficiency
  in optimization,'' 1983.

\bibitem{bubeck2014theory}
S.~Bubeck, ``Theory of convex optimization for machine learning,'' \emph{arXiv
  preprint arXiv:1405.4980}, vol.~15, 2014.

\bibitem{juditsky2022unifying}
A.~Juditsky, J.~Kwon, and {\'E}.~Moulines, ``Unifying mirror descent and dual
  averaging,'' \emph{Mathematical Programming}, pp. 1--38, 2022.

\bibitem{mertikopoulos2018optimistic}
P.~Mertikopoulos, B.~Lecouat, H.~Zenati, C.-S. Foo, V.~Chandrasekhar, and
  G.~Piliouras, ``Optimistic mirror descent in saddle-point problems: Going the
  extra(-gradient) mile,'' in \emph{International Conference on Learning
  Representations}, 2019.

\bibitem{kim2016qualitative}
D.~S. Kim, P.~T. Vuong, and P.~D. Khanh, ``Qualitative properties of strongly
  pseudomonotone variational inequalities,'' \emph{Optimization Letters},
  vol.~10, no.~8, pp. 1669--1679, 2016.

\bibitem{liu2012one}
Q.~Liu, Z.~Guo, and J.~Wang, ``A one-layer recurrent neural network for
  constrained pseudoconvex optimization and its application for dynamic
  portfolio optimization,'' \emph{Neural Networks}, vol.~26, pp. 99--109, 2012.

\bibitem{gao2020continuous}
B.~Gao and L.~Pavel, ``Continuous-time discounted mirror descent dynamics in
  monotone concave games,'' \emph{IEEE Transactions on Automatic Control},
  vol.~66, no.~11, pp. 5451--5458, 2020.

\bibitem{jiang2021game}
Z.~Jiang and J.~Cai, ``Game theoretic control of thermal loads in demand
  response aggregators,'' in \emph{2021 American Control Conference
  (ACC)}.\hskip 1em plus 0.5em minus 0.4em\relax IEEE, 2021, pp. 4141--4147.

\bibitem{tatarenko2018learning}
T.~Tatarenko and M.~Kamgarpour, ``Learning generalized {Nash} equilibria in a
  class of convex games,'' \emph{IEEE Transactions on Automatic Control},
  vol.~64, no.~4, pp. 1426--1439, 2018.

\bibitem{galbis2012vector}
A.~Galbis and M.~Maestre, \emph{Vector analysis versus vector calculus}.\hskip
  1em plus 0.5em minus 0.4em\relax Springer Science \& Business Media, 2012.

\bibitem{durrett2019probability}
R.~Durrett, \emph{Probability: theory and examples}.\hskip 1em plus 0.5em minus
  0.4em\relax Cambridge university press, 2019, vol.~49.

\bibitem{robbins1971convergence}
H.~Robbins and D.~Siegmund, ``A convergence theorem for non-negative almost
  supermartingales and some applications,'' in \emph{Optimizing methods in
  statistics}.\hskip 1em plus 0.5em minus 0.4em\relax Elsevier, 1971, pp.
  233--257.

\bibitem{chung1954stochastic}
K.~L. Chung, ``On a stochastic approximation method,'' \emph{The Annals of
  Mathematical Statistics}, pp. 463--483, 1954.

\bibitem{polyakintroduction}
B.~T. Polyak, ``Introduction to optimization. 1987,'' \emph{Optimization
  Software, Inc, New York}.

\end{thebibliography}

\end{document}